\documentclass{amsart}
\usepackage{amssymb,amsmath}
\usepackage{amsfonts}
\usepackage{hyperref}
\usepackage[notref,notcite]{showkeys}

\newtheorem{theorem}{Theorem}[section]
\theoremstyle{plain}

\newtheorem{lemma}[theorem]{Lemma}
\newtheorem{proposition}[theorem]{Proposition}
\newtheorem{definition}[theorem]{Definition}
\theoremstyle{remark}
\newtheorem{remark}[theorem]{Remark}
\numberwithin{equation}{section}

\DeclareMathOperator{\Real}{Re}

\begin{document}

%%%%%%%%%%%%%%%%%%
%
%		Macros, Shorthand, Etc.
%
%%%%%%%%%%%%%%%%%%

\newcommand{\fbar}{\overline{f}}
\newcommand{\kbar}{{\overline{k}}}
\newcommand{\rbar}{\overline{r}}
\newcommand{\ubar}{\overline{u}}
\newcommand{\wbar}{\overline{w}}
\newcommand{\zbar}{{\overline{z}}}

\newcommand{\calBbar}{\overline{\mathcal{B}}}

\newcommand{\psibar}{\overline{\psi}}

\newcommand{\dee}{\partial}
\newcommand{\dbar}{\overline{\partial}}
\newcommand{\mubar}{\overline{\mu}}
\newcommand{\zetabar}{\overline{\zeta}}

\newcommand{\bbR}{\mathbb{R}}
\newcommand{\bbC}{\mathbb{C}}

\newcommand{\calB}{\mathcal{B}}
\newcommand{\calF}{\mathcal{F}}
\newcommand{\calI}{\mathcal{I}}
\newcommand{\calR}{\mathcal{R}}
\newcommand{\calS}{\mathcal{S}}

\newcommand{\bigO}[1]{{\mathcal{O}}\left( {#1} \right)}
\newcommand{\norm}[1]{\left\| {#1} \right\|}
\newcommand{\angles}[1]{\left\langle {#1} \right\rangle}

\newcommand{\rarr}{\rightarrow}
\newcommand{\darr}{\downarrow}

\newcommand{\eps}{\varepsilon}

\newcommand{\dotarg}{\, \cdot \, }
\newcommand{\compose}{\hspace{0.2mm} \circ \hspace{0.2mm}}

\newcommand{\dint}{\displaystyle{\int}}

\title[Well-Posedness for the DS-II\ Equation]{Global Well-Posedness and Long-Time Asymptotics for the Defocussing
Davey-Stewartson II\ Equation in $H^{1,1}({\bbC})$}
\author{Peter A. Perry}
\address[Christ]{Department of Mathematics, University of California, Berkeley,
Berkeley, California 94720-3840}
\address[Perry]{Department of Mathematics, University of Kentucky, Lexington, Kentucky 40506--0027}
\thanks{Peter Perry supported in part by NSF Grants DMS-0710477 and DMS-1208778}
\thanks{Michael Christ supported in part by NSF grant DMS-0901569.}
\dedicatory{With an Appendix by Michael Christ}
\date{\today}

\begin{abstract}
We use the $\dbar$-inverse scattering method to obtain global
well-posedness and large-time asymptotics for the defocussing Davey-Stewartson
II\ equation. We show that these global solutions are dispersive by computing
their leading asymptotic behavior as $t \rarr \infty$ in terms of an
associated linear problem. These results appear to be sharp.

\end{abstract}
\maketitle
\tableofcontents

\section{Introduction}

In this paper we will use the inverse scattering method to prove global well-posedness for the defocussing
Davey-Stewartson II (DS II) equation
\begin{align}
\label{DSII}
iu_{t}+2\left(  \dbar^{2}+\dee^{2}\right)  u+\left(g+\overline{g}\right)  u  
	&  =0\\
\nonumber
\overline{\partial}g+\partial\left(  \left\vert u\right\vert ^{2}\right)   
	&=0,
\end{align}
a nonlinear, completely integrable dispersive equation in two space dimensions. Here and in what follows, $z=x_1+ix_2$ and
$$ 
\dbar =  \frac{1}{2}\left( \frac{\dee}{\dee x_1} + i \frac{\dee}{\dee x_2} \right), \quad
\dee  =  \frac{1}{2}\left( \frac{\dee}{\dee x_1} - i \frac{\dee}{\dee x_2} \right).
$$

The defocussing DS II equation may be regarded as a two-dimensional analogue of the defocussing cubic nonlinear Schr\"{o}dinger equation in one space dimension: it is one of a multiparameter family of models proposed by Benny-Roskes \cite{BR:1969} and Davey-Stewartson \cite{DS:1974} to model the propagation of weakly nonlinear surface waves in shallow water (see Ghidaglia-Saut \cite{GS:1990} for a physical derivation and extensive local well-posedness results).

We will prove that the Cauchy problem for \eqref{DSII} is globally well-posed for initial data in the space 
$H^{1,1}(\bbC)$. Here and in what follows, $H^{\alpha,\beta}(\bbC)$ denotes the weighted Sobolev space
$$
H^{\alpha,\beta}(\bbC) = 
	\left\{ f \in L^2(\bbC): \angles{D}^\alpha f,\hspace{0.3cm} \angles{\dotarg}^\beta f(\dotarg) \in L^2(\bbC) \right\}. 
$$
Here $\angles{z}=(1+|z|^2)^{1/2}$ and $\angles{D}^\beta$ is the Fourier multiplier with symbol $\angles{\xi}^\beta$.

We will also show, under stronger conditions on the initial data, that solutions are asymptotic in $L^\infty$-norm to solutions of the linear problem
\begin{equation}
\label{DSII.linear}
iv_t + 2 \left( \dbar^2 + \dee^2 \right) v = 0.
\end{equation}

To state our result precisely, we recall the formulation of \eqref{DSII} as a nonlinear integral equation.
Denote by $S(t)$ the solution operator for the linear problem \eqref{DSII.linear}. For $T>0$, let
$$
X_{T}=C([0,T],L^{2}({\bbC}))\cap L^{4}([0,T]\times{\bbC})
$$
We say that a function $u\in X_{T}$ solves the Davey-Stewartson II equation
with initial data $u_{0}\in L^{2}({\bbC})$ if $u(t)$ solves the
equation
\begin{equation}
\label{DSII.int}
u(t)=S(t)u_{0}+\Lambda(u)(t), 
\end{equation}
for $t\in(0,T)$. Here
\[
\Lambda\left(  u	\right)  (t)
	=i\int_{0}^{t}S(t-s)
		\left(  
			2u(s)\Real 
					\left[
						 \calS 
								\left(  \left\vert u(s)\right\vert ^{2}	\right)  
					\right]
		\right)  \, ds.
\]
where $\calS = \dee \dbar^{-1}$ is the Beurling transform (see Lemma \ref{lemma:S}). Strichartz estimates for $S(t)$ (see \cite[\S 2 and Appendix]{GS:1990}) show that the solution operator $S(t)$ and the nonlinear mapping $\Lambda$ takes $X_T$ to itself, so that \eqref{DSII.int} can be formulated as a fixed-point problem in this space. It is not difficult to see that a classical solution of \eqref{DSII} belonging to $C^1([0,T],\calS(\bbC))$ also solves \eqref{DSII.int}. Ghidalgia and Saut \cite[Theorem 2.1]{GS:1990} showed that for initial data $u_0 \in L^2(\bbC)$, the problem \eqref{DSII.int} has a solution in $X_T$ for some $T>0$ depending on the initial data. For $u_0 \in H^{1,1}(\bbC)$ we can globalize this result by the inverse scattering method. We will prove:

\begin{theorem}
\label{thm:global} There exists a continuous
map
\begin{align*}
H^{1,1}({\bbC})\times\bbR  &  \rarr H^{1,1}(\bbC)\\
\left(  u_{0},t\right)   &  \mapsto u(t)
\end{align*}
so that the function $u$ is a solution of the Davey-Stewartson II equation
\eqref{DSII} with initial data $u_0$  in the sense that the integral equation \eqref{DSII.int}
holds for all $t$. Moreover, $\left\Vert u(t)\right\Vert _{2}$ is conserved.
\end{theorem}

Since $H^{1,1}({\bbC})\subset L^{p}({\bbC})$ for
all $p\in\left(  1,\infty\right)  $ (see \eqref{H11.Lp})  it is easy to see that $C([0,T],H^{1,1}(\bbC))$ is continuously embedded in $X_T$. Hence, the global
solution constructed in Theorem \ref{thm:global} coincides with the local Ghidaglia-Saut solution for all
$T$, so that these solutions extend to $T=\infty$ when $u_{0}\in H^{1,1}({\bbC})$.

Our proof exploits the completely integrable method for the defocussing DS\ II
equation developed by Fokas \cite{Fokas:1983}, Ablowitz-Fokas
\cite{AF:1982,AF:1983,AF:1984}, Beals-Coifman \cite{BC:1985,BC:1989,BC:1990},
Sung \cite{Sung:1994}, and Brown \cite{Brown:2001}. For $u_{0}\in
\calS\left(  {\bbC}\right)  $, the function
\begin{equation}
 \label{ISM1}
u(z,t)=\calI \left[  e^{4it\Real \left( \left(\dotarg\right)  ^{2}\right)  }
								\left(  \calR u_{0}\right)  \left( \dotarg \right)  
				\right]  (z)
\end{equation}
solves the Cauchy problem for (\ref{DSII}) with initial data $u_{0} \in \calS(\bbC)$ 
(see Appendix \ref{sec:time} for a
self-contained proof and for references to the literature). Here $\calR$
and $\calI$ are the direct and inverse scattering transforms for the
DS II equation which we now describe.

For $z=x_1+ix_2$ and $k=k_1+ik_2$, let $e_k(z)$ be the unimodular function
$$ e_k(z) = e^{\kbar \zbar - kz} =\exp(-2i(k_1x_2+k_2x_1)).$$
The \emph{direct scattering map}  $\calR$ is defined by the $\dbar$-problem (in the $z$-variable)
\begin{subequations}
		\label{mu.dbar}
		\begin{align}
		\label{mu.dbar.1}
		\dbar \mu_1 &= \frac{1}{2} e_k u  \overline{\mu_2},\\
		\label{mu.dbar.2}
		\dbar \mu_2 &= \frac{1}{2} e_k u \overline{\mu_1},\\
		\label{mu.dbar.3}
		\lim_{|z| \rarr \infty} (\mu_1(z,k),\mu_2(z,k)) &= (1,0)
		\end{align}
\end{subequations}
and the representation formula
\begin{equation}
		\label{R}
		(\calR u)(k) = \frac{1}{\pi} \int e_k(z) u(z) \overline{\mu_1(z,k)} \, dA(z)
\end{equation}
(here and in what follows, $dA$ denotes Lebesgue measure on $\bbC$). Given $u \in \calS(\bbC)$ one first solves \eqref{mu.dbar} for $\mu_1,\mu_2$, and then computes 
$\calR u$ from \eqref{R}. The linearization of the map $\calR$ at $u=0$ is the map
\begin{equation}
\label{F}
\left( \calF f \right)(k) = \frac{1}{\pi} \int e_k(z) f(z) \, dA(z)
\end{equation}
which is the usual two-dimensional Fourier transform up to a linear change of variables.

The \emph{inverse scattering map} $\calI$ is similarly defined by the $\dbar$-problem
\begin{subequations}
		\label{nu.dbar}
		\begin{align}
		\label{nu.dbar.1}
		\dbar_k \nu_1 &= \frac{1}{2} e_k \rbar \,  \overline{\nu_2}\\
		\label{nu.dbar.2}
		\dbar_k \nu_2 &= \frac{1}{2} e_k \rbar  \, \overline{\nu_1}\\
		\label{nu.dbar.3}
		\lim_{|k| \rarr \infty} (\nu_1(z,k),\nu_2(z,k)) &= (1,0)
		\end{align}
\end{subequations}
and the representation formula
\begin{equation}
		\label{I}
		(\calI r)(z) = \frac{1}{\pi} \int e_{-k}(z) r(k) \nu_1(z,k) \, dA(k)
\end{equation}
Here $\dbar_k$ denotes the $\dbar$-operator acting in the $k$ variable. 
Given $r \in \calS(\bbC)$, one first solves \eqref{nu.dbar} for $\nu_1,\nu_2$, and then computes $\calI r$ from \eqref{I}.
The linearization of the map $\calI$ at $r=0$ is the inverse Fourier transform
\begin{equation}
\label{Finv}
\left( \calF^{-1} g\right) (z) = \frac{1}{\pi} \int e_{-k}(z) g(k) \, dA(k).
\end{equation}
From the definitions it is formally obvious that 
\begin{equation}
\label{IR}
\calI = C \compose \calR \compose C
\end{equation}
 where $C$ is complex conjugation. This fact, proved in Lemma \ref{lemma:IR} of what follows (see also \cite[\S 2]{AFR:2015}), will allow us to apply our analysis of $\calR$ directly to $\calI$.
		
The solution formula \eqref{ISM1} for \eqref{DSII} should be compared to the Fourier transform solution 
formula
\begin{equation}
\label{FSM1}
v(z,t)=\calF^{-1} \left[ 
									e^{4it \Real
											\left( 
													\left(\dotarg\right)^2
											\right)}
											 \left(\calF v_0 \right)(\dotarg) 
						\right](z) 
\end{equation}
for the linearized problem \eqref{DSII.linear}. Using the  definition \eqref{nu.dbar}--\eqref{I} of the map $\calI$, we can recast the solution formula \eqref{ISM1} for the DS II equation as a $\dbar$-problem depending on space and time as parameters.  Given $u_0 \in \calS(\bbC)$, one computes $r_0 = \calR u_0$ and solves the $\dbar$-problem

\begin{align}
\label{nu.dbar.t}
\dbar_k \nu_1 	&= 	\frac{1}{2} e^{-itS} \overline{r_0 \nu_2}, \\
\nonumber
\dbar_k \nu_2		&=	\frac{1}{2} e^{-itS} \overline{r_0 \nu_1}, \\
\nonumber
\lim_{|k| \rarr \infty} (\nu_1,\nu_2) &= (1,0)	.
\end{align}
Here
\begin{equation}
\label{S}
S(z,k,t) = \frac{kz-\kbar\zbar}{it} + 4 \Real \left( k^2 \right)
\end{equation}
is a real-valued phase function with a single nondegenerate critical point 
\begin{equation}
\label{S.crit}
k_c = iz/4t.
\end{equation}
We then recover the solution from the formula
\begin{equation}
\label{ISM2}
u(z,t) = \frac{1}{\pi} \int e^{itS(z,k,t)} r_0(k) \nu_1(z,k,t) \, dA(k).
\end{equation}

By a careful study of the $\dbar$-problems \eqref{mu.dbar} and \eqref{nu.dbar}, we will prove:

\begin{theorem}
\label{thm:lip}
The maps $\calR$ and $\calI$, initially defined on $\calS(\bbC)$ by \eqref{mu.dbar}--\eqref{R} and \eqref{nu.dbar}--\eqref{I}, extend to locally Lipschitz continuous maps from $H^{1,1}(\bbC)$ to itself. Moreover, $\calR  \compose \calI = \calI \compose \calR= I$ where $I$ is the identity mapping on $H^{1,1}(\bbC)$. Finally, the Plancherel relations $\norm{\calR u}_2 = \norm{r}_2$ and $\norm{\calI r}_2=\norm{r}_2$ hold.
\end{theorem}

\begin{proof}[Proof of Theorem \ref{thm:global}, given Theorem \ref{thm:lip}]
First, we show that
the map defined by \eqref{ISM1} has the claimed continuity properties.
For $u_{1}$ and $u_{2}$ in a fixed bounded subset of $H^{1,1}({\bbC})$
and $t,t^{\prime}>0$, let
\begin{align*}
U_{1}(z,t)  &  =\calI	\left[  
								e^{4it\Real \left(  \left( \dotarg\right) ^{2} \right)  }
									\calR\left(  u_{1}\right)  
							\right]
					\left(  z\right)  ,\\
U_{2}(z,t^{\prime})  &  =
					\calI	\left[  
								e^{4it^{\prime}\Real \left(  \left(  \dotarg\right)  ^{2}\right)  }
									\calR\left(u_{2}\right)  
							\right]  
							\left(  z\right)  .
\end{align*}
Then
\begin{align*}
\left\Vert U_{1}(\dotarg,t)-U_{2}(\dotarg,t^{\prime})\right\Vert _{H^{1,1}}
		&  \leq C\left\Vert 
					e^{4it\Real\left(  \left(  \dotarg\right)^{2}\right)  }
							\calR\left(  u_{1}\right)  -
					e^{4it^{\prime}\Real\left(  \left(  \dotarg\right)  ^{2}\right)  }
							\calR\left(  u_{2}\right)  
					\right\Vert _{H^{1,1}}\\
		&  \leq C\left\Vert 
					e^{4it\Real\left(  \left(  \dotarg\right)^{2}\right)  }
						\left[  \calR\left(  u_{1}\right)  -\calR\left(u_{2}\right)  \right]  
					\right\Vert _{H^{1,1}}\\
		&  +C\left\Vert
					 \left(  e^{4it\Real\left(  \left(  \dotarg\right)  ^{2}\right)  }-
					 e^{4it^{\prime}\Real\left(  \left(\dotarg\right)  ^{2}\right)  }\right)  
					 \calR\left(  u_{2}\right)
					\right\Vert _{H^{1,1}},
\end{align*}
where $C$ is uniform in $u_1$ and $u_2$ in a fixed bounded subset of $H^{1,1}(\bbR)$.
The continuity now follows from the Lipschitz continuity of $\calR$, the
estimate 
$$
\left\Vert 
		e^{4it\Real\left(  \left(  \dotarg\right)^{2}\right)  }f
\right\Vert _{H^{1,1}}
\leq 
	C\left(  1+\left\vert t\right\vert\right)  
			\left\Vert f\right\Vert _{H^{1,1}}
$$ 
and the fact that
\[
\lim_{\left\vert t\right\vert \rarr0}
		\left\Vert 
			\left[ 
				e^{4it\Real\left(  \left(  \dotarg\right)  ^{2}\right)}-1
			\right]  
			f
		\right\Vert _{H^{1,1}}=0
\]
for each fixed $f\in H^{1,1}({\bbC})$ by dominated convergence.

Next we prove that the map \eqref{ISM1} solves the DS\ II\ equation
(\ref{DSII.int}) for  initial data 
$u_{0}\in H^{1,1}(\bbR^{2})$. For $u_{1}$ and $u_{2}$ in a fixed bounded subset $B$ of
$H^{1,1}({\bbC})$ and $T>0$, we have 
\begin{equation}
\sup_{t\in\left[  0,T\right]  }
	\left\Vert  U_1(\dotarg,t) - U_2(\dotarg,t)  \right\Vert_{H^{1,1}}
\leq C\left\Vert u_{1}-u_{2}\right\Vert _{H^{1,1}}
\label{eq:usta}
\end{equation}
where $C=C(T,B)$, by Theorem \ref{thm:lip}. Now let 
$u_{0}\in H^{1,1}({\bbC})$ be given and let $\left\{  u_{n,0}\right\}  _{n=1}^{\infty}$
be a sequence from $\calS\left(  {\bbC}\right)  $ with $u_{n,0}\rarr u_{0}$ in $H^{1,1}({\bbC})$. Let
\[
U_{n}\left(  z,t\right)  =
	\calI\left(  e^{4it\Real\left((\dotarg)^{2}\right)  }
	\calR\left(  u_{n,0}\right)  \right)
\]
and
\[
U(z,t)=\calI
		\left(  e^{4it\Real\left(  (\dotarg)^{2}\right)  }
		\calR\left(  u_{0}\right)  \right)  .
\]
By (\ref{eq:usta}),%
\[
\sup_{t\in\left[  0,T\right]  }
	\left\Vert U_{n}(\dotarg,t)-U(\dotarg,t)\right\Vert _{H^{1,1}}
	\leq C\left\Vert u_{n,0}-u_{0}\right\Vert _{H^{1,1}}
\]
so that, in particular, 
$U_{n}\rarr U$ in $X_T$. 
Since $u_{n,0}\in\calS\left(  {\bbC}\right) $, we have 
$U_{n}\in C\left(  \left[0,T\right]  ,\calS\left(  {\bbC}\right)  \right)  $, and each
$U_{n}$ satisfies
\begin{equation}
U_{n}(t)=S(t)u_{n,0}+\Lambda(U_{n})(t). 
\label{eq:DS.int.n}
\end{equation}
Since
\[
\left\Vert U-U_{n}\right\Vert _{X_{T}}
	\leq C(T) \sup_{t\in\left[  0,T\right] }
					\left\Vert 
						U(\dotarg,t)-U_{n}(\dotarg,t)
					\right\Vert _{H^{1,1}(\bbC)}
\]
and $\Lambda$ is a continuous mapping from $X_{T}$ to itself, it follows that
$\Lambda(U_{n}) \rarr \Lambda\left(  U\right)  $ in $X_{T}$. Taking limits
in (\ref{eq:DS.int.n}) in the $X_T$-topology, we conclude that
\[
U(t)=S(t)u_{0}+\Lambda(U)(t)
\]
so that $U$ solves the DS\ II\ equation \eqref{DSII.int} with initial data $u_{0}$.
\end{proof}

Through a careful study of the $\dbar$-problem \eqref{nu.dbar.t}, we will prove:

\begin{theorem}
\label{thm:asy}Suppose that $u_{0} \in H^{1,1}({\bbC}) \cap L^1({\bbC})$. The solution $u$ of the defocussing DS\ II\ equation with Cauchy data $u_{0}$ obeys the asymptotic formula
\[
u(z,t)=v(z,t)+o\left(  t^{-1}\right)
\]
in $L_{z}^{\infty}$-norm, where
\[
v(z,t)=\calF^{-1}\left(  e^{4it\Real \left(  \left(
\dotarg\right)  ^{2}\right)  }\left( \calR u_0 \right)(\dotarg)\right)
\]
\end{theorem}

\begin{remark}In an earlier version of this paper, the hypothesis that $u _0 \in L^1({\bbC})$ was erroneously omitted. The condition $u_0 \in L^1({\bbC}) $ implies that $r_0$ is continuous (see Remark \ref{rem:r}).  The additional hypothesis appears to be necessary for the proof: see Lemma \ref{lemma:I4} for the key step where the continuity of $r_0$  is used.
\end{remark}

\begin{remark}
This result shows that, in contrast to the one-dimensional cubic nonlinear Schr\"{o}dinger equation, there is no ``logarithmic phase shift'' in the solution due to the nonlinear term. See Deift-Zhou \cite{DZ:2003} for an analysis of this phenomenon and for references to the literature.
\end{remark}

\begin{remark}
Suppose that $r_{0}$ is continuous and that $\calF^{-1} r_0 \in L^1({\bbC})$. This assumption holds, for example, when $u_0 \in \calS({\bbC})$, so that $r_0 \in \calS({\bbC})$ by Sung's work \cite[Paper II, \S 4]{Sung:1994} on the scattering transform.  The function $v(z,t)$ is given by
\[
v(z,t)=\int\Gamma_{t}\left(  z-z^{\prime}\right)  \left(  \calF
^{-1}r_{0}\right)  (z^{\prime})~dA(z^{\prime})
\]
where
$$
\Gamma_{t}(z)=\frac{e^{i\left(  z^{2}+\zbar^{2}\right)  /8t}}{4t}.
$$
From this formula, we obtain
\[
\lim_{t\rarr\infty}\Gamma_{t}(z)^{-1}v(z,t)=\int\left(  \calF
^{-1}r_{0}\right)  (z^{\prime})~dA(z^{\prime})=\pi r_{0}(0)
\]
which shows that the remainder $o\left(  t^{-1}\right)  $ is indeed of lower
order provided $r_{0}\left(  0\right)  \neq0$.
\end{remark}

The results of Theorem \ref{thm:asy} were first obtained by Kiselev
\cite{Kiselev:1997} (see also \cite[Theorem 7]{Kiselev:2004}). On the one
hand, Kiselev's result treats both the focusing and defocussing
DS\ II\ equations; on the other, he imposes a \textquotedblleft small
data\textquotedblright\ restriction and more stringent integrability and
regularity assumptions. Kiselev's analysis relies in part on separate
asymptotic expansions of the solution $\nu_{1}(z,k,t)$ in the `exterior
region' $\left\vert k-k_{c}\right\vert \geq t^{-1/4}$ and in the `interior
region' $\left\vert k-k_{c}\right\vert <2t^{1/4}$ with matching in the
transition region.

In our proof, we remove Kiselev's small data restriction in the defocussing
case and replace the asymptotic expansions with a finer analysis of the
integral operator $M$ (see \eqref{M}) used to solve \eqref{nu.dbar.t}.
Our analysis rests on scaling arguments and on the simple integration by parts
formula \eqref{bug} previously used by Bukhgeim \cite{Bukhgeim:2008} in
his analysis of the inverse conductivity problem.

Inverse scattering for the defocussing Davey-Stewartson II equation was studied by Fokas
\cite{Fokas:1983}, Ablowitz-Fokas \cite{AF:1982,AF:1983,AF:1984},
Beals-Coifman \cite{BC:1985,BC:1989,BC:1990}, Sung \cite{Sung:1994}, and Brown
\cite{Brown:2001}. Beals and Coifman construct global solutions for the
defocussing DS\ II\ equation with initial data in $\calS(\bbC)$ by inverse scattering methods, while Sung constructs solutions for the
same case if $u\in L^{1}\cap L^{\infty}$ and the Fourier transform of $u$ lies
in $L^{1}\cap L^{\infty}$ (see paper III\ of \cite{Sung:1994}). Sung
\cite{Sung:1995} obtained the leading $t^{-1}$ decay rate (but not the
asymptotic formula)\ for solutions of the DS\ II\ equation with Schwarz class
initial data, using his earlier work \cite{Sung:1994} on the inverse
scattering method for DS\ II. Sung and Brown construct solutions to the
focusing DS\ II\ equations with small initial data; the small data hypothesis
avoids soliton solutions (see \cite{APP:1989} and see section 10.5 of \cite{DL:2007} for an exposition and additional references) and the blow-up phenomena discussed below. Brown actually shows Lipschitz continuity of the solution map for
\eqref{DSII.linear} for small Cauchy data in $L^{2}$ for either the
focusing or defocussing DS\ II\ equation. More recently, Astala, Faraco, and Rogers \cite{AFR:2015} proved Lipschitz continuity of the scattering map $\calR$ from $H^{s,s}$ to $L^2$ for $s \in (0,1)$ and proved a Plancherel identity for $\calR$.

The same analysis used here can also be applied to the focusing
DS\ II equation with small initial data, which differs from (\ref{DSII}) in
that the second equation reads%
\[
\dbar g-\dbar\left(  \left\vert u\right\vert
^{2}\right)  =0,
\]
changing the sign of the nonlinear term. The \textquotedblleft small
data\textquotedblright\ condition is used to replace the Fredholm argument in Lemma \ref{lemma:T.Fred}. Details will be given in \cite{Perry:2015}. Ozawa
\cite{Ozawa:1992} constructed a solution to the focusing DS\ II\ equation
with the following properties: \ (1)\ the initial data $u_{0}\in L^{2}$, but
$\left\vert \nabla u_{0}(z)\right\vert ,~\left\vert zu(z)\right\vert \geq
C(1+\left\vert z\right\vert )^{-1}$ for a positive constant $C$, (2)\ the
measure $\left\vert u(z,t)\right\vert ^{2}dA(z)$ concentrates to a $\delta
$-function in finite time (see also C. Sulem and P. Sulem \cite{SS:1999}, pp.
229-230). Since $\nabla u_{0}$ and $\left(  \dotarg \right)  u_{0}(\dotarg)$
lie in weak-$L^{2}$ but not $L^{2}$, Ozawa's results suggest that
$H^{1,1}({\bbC})$ is a natural limit for the inverse scattering method.

In \cite{Perry:2013}, we use the results of this paper and previous work of Lassas, Mueller, and Siltanen \cite{LMS:2007} and Lassas, Mueller, Siltanen, and Stahel \cite{LMSS:2012} to find global solutions of the Novikov-Veselov equation with initial data of conductivity type by the inverse scattering method. In \cite{BOPS:2015}, co-authored with Russell Brown, Katharine Ott, and Nathan Serpico, we show that the maps $\calR$ and $\calI$ have mapping properties between weighted Sobolev spaces which parallel those of the Fourier transform. Our analysis in \cite{BOPS:2015} relies in part on a key estimate of Astala, Faraco, and Rogers \cite{AFR:2015} that generalizes our Lemma \ref{lemma:Tk.decay}. These authors prove a Plancherel formula for the map $\calR$ under less restrictive hypotheses than ours. 

We close by sketching the contents of this paper. In \S \ref{sec:prelim}, we fix notation,  recall basic facts about integral operators associated to the $\dbar$-problem, recall key Brascamp-Lieb inequalities, and prove an important lemma on integration by parts. 
In \S \ref{sec:osc}, we study the $\dbar$-problem \eqref{mu.dbar} in depth. We apply these results in \S \ref{sec:maps} to prove Theorem \ref{thm:lip}. Finally, we prove Theorem \ref{thm:asy} in \S \ref{sec:large.t}.
Appendix A, written by Michael Christ, proves Brown's multilinear estimate (Proposition \ref{brown} and \cite[Lemma 3]{Brown:2001}) by the methods of Bennett, Carbery, Christ, and Tao \cite{BCCT1,BCCT2}. In Appendix B we present a concise proof that the inverse scattering formula \eqref{ISM1} gives a classical solution of the DS II equation for initial data in $\calS(\bbC)$. Appendix C computes large-$z$ asymptotic expansions for solutions of \eqref{nu.dbar} that are used in Appendix B.

\emph{Acknowledgments}. 
It is a pleasure to thank Russell Brown, Ken
McLaughlin, Michael Music, and Peter Topalov for helpful discussions, to thank
Russell Brown, Peter Miller, and Katharine Ott for a careful reading of the manuscript, and to
thank Michael Christ and Catherine Sulem for helpful correspondence. I\ am
also grateful to the referee for an exceptionally thorough and meticulous reading of three (!) versions of this manuscript, for
pointing out several errors in an earlier version of this paper, and for
numerous helpful suggestions which have considerably improved the manuscript. The current proof of Lemma \ref{lemma:IR} incorporates a suggestion of the referee.
Part of this work was carried out at the
Mathematical Sciences Research Institute in Berkeley, California, whose
hospitality the author gratefully acknowledges.

\section{Preliminaries}
\label{sec:prelim}

\emph{Notation, function spaces}.
We denote by $C^0(\bbC)$ the bounded continuous functions on $\bbC$ equipped with the sup norm, and by $C_0(\bbC)$ the continuous functions that vanish at infinity.
The spaces $L^p(\bbC)$ are the usual Lebesgue spaces and $p'$ the H\"{o}lder conjugate exponent. We sometimes write $L^p_z(\bbC)$ or $L^p_k(\bbC)$ to clarify the choice of integration variable $z$ or $k$. The space $L^{2,1}(\bbC)$ consists of complex-valued measurable functions $f\in L^2(\bbC)$ with $\angles{z} f \in L^2(\bbC)$. We denote by $\angles{f,g}$ the dual pairing
\begin{equation}
\label{pairing}
\angles{f,g} = \frac{1}{\pi} \int \overline{f(z)} g(z) \, dA(z).
\end{equation}

To quantify regularity of solutions for \eqref{mu.dbar} and \eqref{nu.dbar}, we use the usual H\"{o}lder spaces. For $\alpha \in (0,1)$, let $C^\alpha$ denote the bounded, H\"{o}lder continuous functions of order $\alpha$ on $\bbC$ equipped with the norm
$$ \norm{f}_{C^\alpha} = \| f \|_\infty + \sup_{z \neq z' } \frac{\left|f(z)-f(z')\right|}{|z-z'|^\alpha} . $$

If $X$ and $Y$ are Banach spaces, $\calB(X,Y)$ (resp. $\calBbar(X,Y)$) is the Banach space of linear (resp.\  antilinear) operators from $X$ to $Y$. We write $\calB(X)$ for $\calB(X,X)$, and similarly for $\calBbar(X)$. 

The space $H^{1,1}(\bbC)$  is continuously embedded in $L^p(\bbC)$ for any $p \in (1,\infty)$.  Thus, for any $s \in (1,\infty)$,
\begin{equation}
\label{H11.Lp}
\norm{u}_s \leq C_s \norm{u}_{H^{1,1}}
\end{equation}
where $C_s$ depends only on $s$. We also have the following standard compact embedding result. We give a proof since we have not found a reference although the result is well-known.
\begin{lemma}
\label{lemma:H11.Lp}
The space $H^{1,1}(\bbC)$ is compactly embedded in $L^p(\bbC)$ for any $p\in (1,\infty)$. 
\end{lemma}
		
\begin{proof} 
Observe that $\calF$ preserves $H^{1,1}(\bbC)$ and maps $L^p(\bbC)$ continuously to $L^{p'}(\bbC)$ for $p \in [1,2]$. Hence, if we show that $H^{1,1}(\bbC)$ is compactly embedded in $L^p(\bbC)$ for $p \in (1,2]$, the same fact for $p \in [2,\infty)$ is follows by composing with the continuous map $\calF$.  

To show that $H^{1,1}(\bbC)$ is compactly embedded in $L^p(\bbC)$ for $p \in (1,2]$, let $\eta \in C_0^\infty(\bbC)$ with $\eta(w)=1$ for $|w| \leq 1$ and $\eta(w)=0$ for $|w| \geq 2$, and let $\eta_R(z)=\eta(w/R)$. Let $T_R f = \eta_R \cdot f$. Since $T_R$ is a bounded map from $H^{1,1}(\bbC)$ to $W^{1,p}(\bbC)$ for any $p \in [1,2]$, it follows from the Rellich-Kondrakov Theorem that $T_R$ is a compact mapping from $H^{1,1}(\bbC)$ to $L^q(\bbC)$ for any $q \in [1,\infty)$. For $q \in (1,2]$ we have by H\"{o}lder's inequality that $\norm{(I-T_R) f}_q \leq C_q  R^{\frac{2(1-q)}{q}} \norm{f}_{L^{2,1}}$, so that $\norm{(I-T_R)}_{\calB(L^q)}$
vanishes as $R \rarr \infty$. The compact embedding now follows from norm-closure of the compact operators.
\end{proof}

\emph{Estimates and Vanishing Theorem for the $\dbar$-problem}.
The solid Cauchy transform is given by 
$$ (Pf)(z) = \frac{1}{\pi} \int \frac{1}{z-\zeta} f(\zeta) \, dA(\zeta) $$
and is an inverse for the $\dbar$-operator in the sense that, for $f \in C_0^\infty(\bbC)$,
\begin{equation}
\label{P0}
P (\dbar f) = \dbar ( Pf) = f.
\end{equation}
Results analogous to those described below also hold for the operator
$$ (\overline{P}f)(z) = \frac{1}{\pi} \int \frac{1}{\zbar-\zetabar} f(\zeta) \, dA(\zeta) $$
which is an inverse for the $\dee$-operator.

The following estimates extend $P$ to a larger domain. They
are proved, for example, in Vekua \cite[Chapter I.6]{Vekua:1962} or Astala, Iwaniec, and Martin \cite[\S 4.3]{AIM:2009}.

(1) Fractional integration and H\"{o}lder estimates. 
If $q \in (1,2)$ then $\tilde{q}$ denotes the Sobolev conjugate $(\tilde{q})^{-1}=q^{-1}-1/2$. 
It follows from the Hardy-Littlewood-Sobolev inequality that 
\begin{equation}
		\label{HLS}
		\norm{Pf}_{\tilde{q}} \leq C_q \norm{f}_q.
\end{equation}
We usually take $\tilde{q}=p$ and $q=2p/(p+2)$ for $p \in (2,\infty)$. 
From this inequality and H\"{o}lder's inequality we see that for 
$p \in (2,\infty)$,  $v \in L^2(\bbC)$ and $u \in L^p(\bbC)$
\begin{equation}	
		\label{OP}
		\norm{P(vf)}_p \leq C_p \norm{v}_2 \norm{f}_p.
\end{equation}
		
It follows from H\"{o}lder's inequality that for any $q,r$ with $1<q<2<r <\infty$, 
\begin{equation}
	    \label{vekua}
	    \norm{Pf}_\infty \leq C_{q,r} \left(\norm{f}_q+\norm{f}_r \right).
\end{equation}

(2) H\"{o}lder continuity and asymptotic behavior.
For any $p >2$ and $f \in L^p(\bbC) \cap L^{p'}(\bbC)$,
\begin{equation}
		\label{Holder}
		\left| (Pf)(z)-(Pf)(z') \right| \leq C_p | z-z'|^{1-2/p} \norm{f}_p.
\end{equation}
	    
If $p \in (2,\infty)$ and $f \in L^p \cap L^{p'}$ then 
\begin{equation}
	\label{vanish}
	\lim_{|z| \rarr \infty} (Pf)(z)=0.
\end{equation}
	
By \eqref{vekua} and a density argument, it is enough to show that \eqref{vanish} holds for $f \in C_0^\infty(\bbC)$. This is a straightforward computation.

The following lemma will allow us to recast \eqref{mu.dbar}, \eqref{nu.dbar}, and \eqref{nu.dbar.t} as integral equations.
\begin{lemma}
\label{solve}
Suppose $f \in L^q(\bbC)$ for $q \in (1,2)$. A function $u \in L^{\tilde{q}}(\bbC)$ solves $\dbar u = f$ in distribution sense if and only if $u = Pf$.
\end{lemma}

\begin{proof}
For any $f\in L^{q}({\bbC})$, it follows from \eqref{P0} and \eqref{HLS} that $u=Pf$ solves $\dbar u=f$ in distribution sense. 

Suppose, on the other hand, that $f \in L^q(\bbC)$, that $u\in L^{\tilde{q}}({\bbC})$, and that 
$\dbar u=f$ in distribution sense. Let $v=u-Pf$. It follows that $\partial\dbar v=0$ in
distribution sense, so that $v\in C^{\infty}$ by Weyl's lemma. Thus,
$v$ is a holomorphic function belonging to $L^p(\bbC)$, so $v$ vanishes identically by 
Liouville's Theorem.
\end{proof}

The following vanishing theorem is a special case of Brown and Uhlmann \cite[Corollary 3.11]{BU:1997} that will suffice for our purpose.
\begin{lemma} 
\label{Liouville}
Suppose that $w \in L^p(\bbC) \cap L^2_{\mathrm{loc}}(\bbC)$ for some $p \in (1,\infty)$, 
that $a \in L^2(\bbC)$, and that  $\dbar w = a \overline{w}$ in distribution sense. Then $w=0$.
\end{lemma}

\emph{Basic Estimates on the Beurling Transform}.
The Beurling transform $\calS$ is defined on $C_0^\infty(\bbC)$ by 
$$ \left(\calS f\right)(z) = -\frac{1}{\pi} \lim_{\eps \darr 0} \int_{|w-z|>\eps} \frac{1}{(z-w)^2} f(w) \, dA(w) $$
and obeys the relation $\dbar ( \calS f ) =  \dee f$. We refer the reader to \cite[\S 4.3]{AIM:2009} for discussion and proofs.

\begin{lemma}
\label{lemma:S}
The operator $\calS$ extends to a bounded operator from $L^p(\bbC)$ to itself for any $p \in (1,\infty)$, unitary if $p=2$. Moreover, if $\nabla \varphi$ belongs to $L^q(\bbC)$ for some $q \in (1,\infty)$, then 
$\calS(\dbar \varphi) = \dee \varphi$.
\end{lemma}

Thus, if $u \in L^p(\bbC)$ for some $p \in (1,\infty)$ and $\nabla u \in L^q(\bbC)$ for $q \in (1,\infty)$, the norms $\norm{\dbar u}_q$, $\norm{\dee u}_q$, and $\norm{\nabla u}_q$ are mutually equivalent. 

We will also use the analogous results for the transform
$$ \left(\calS^* f\right)(z) = -\frac{1}{\pi} \lim_{\eps \darr 0} \int_{|w-z|>\eps} \frac{1}{(\zbar-\wbar)^2} f(w) \, dA(w) $$
which satisfies $\dee \left(\calS^* f\right) = \dbar f$ on $C_0^\infty(\bbC)$.

\emph{Integration by Parts}.
If $\varphi(z,\theta)$ is a smooth, real-valued function with isolated critical points in the integration variable $z$, and if $f \in C_0^\infty$ with support away from the critical points of $\varphi$, then
\begin{equation}
\label{bug}
	P[ e^{i\varphi} f ](z) = \frac{e^{i\varphi}}{i\varphi_{\zbar}}f(z) 
		- \frac{1}{i} P\left[ e^{i\varphi} \dbar_z\left(\varphi_\zbar^{-1} f\right) \right](z).
\end{equation}
In case $i\varphi(z,k)=\kbar \zbar - kz$ for $k \neq 0$, the phase function $\varphi$ has no critical points. Hence, for any $f \in C_0^\infty(\bbC)$, the identity
\begin{equation}
\label{ip0}
P[e_k f] = \frac{e_k}{\kbar} f - \frac{1}{\kbar}P\left[ e_k \left(\dbar_z f\right) \right]
\end{equation}
holds.

Let $q \in (1,2)$. Approximating $f \in W^{1,q}(\bbC)$ by a sequence from $C_0^\infty(\bbC)$, we can conclude that for $f \in W^{1,q}(\bbC)$, the equality \eqref{ip0}
holds in $L^{\tilde{q}}$. Note that, if $f \in W^{1,q}(\bbC)$, Sobolev embedding implies that $f \in L^{\tilde{q}}$ so the statement makes sense. From this we obtain the estimate
\begin{equation}
\label{ip1}
\norm{P[e_k f]}_{\tilde{q}} \leq C_q \left\langle k \right\rangle^{-1} \norm{f}_{W^{1,q}}.
\end{equation}

\emph{Brascamp-Lieb type estimates}. 
The following multilinear estimate, due
to Russell Brown \cite[Lemma 3]{Brown:2001} (see also Nie-Brown
\cite{BN:2011}), plays a crucial role in the analysis of solutions to
(\ref{mu.dbar}) and (\ref{nu.dbar}). See Appendix \ref{sec:christ} for a
proof of the estimate by the methods of Bennett, Carbery, Christ and Tao
\cite{BCCT1,BCCT2}. Define%
\[
\Lambda_{n}(\rho,u_{0},u_{1},\ldots,u_{2n})=\int_{\bbC^{2n+1}}%
\frac{\left\vert \rho(\zeta)\right\vert \left\vert u_{0}(z_{0})\right\vert
\ldots\left\vert u_{2n}(z_{2n})\right\vert }{\prod_{j=1}^{2n}\left\vert
z_{j-1}-z_{j}\right\vert }dA(z)
\]
where $dA(z)$ is product measure on $\bbC^{2n+1}$ and
$$
\zeta=\sum_{j=0}^{2n}(-1)^{j}z_{j}. 
$$

\begin{proposition}
\cite{Brown:2001} 
\label{brown}
For any functions $\rho,u_{0},u_{1},\ldots,u_{2n}\in L^{2}({\bbC})$, the estimate
\begin{equation}
\label{ineq:brown}
\left\vert \Lambda_{n}(\rho,u_{0},u_{1},\ldots,u_{2n})\right\vert
	 \leq	C_{n}
	 			\left\Vert 
	 					\rho\right\Vert _{2}\prod_{j=0}^{2n}\left\Vert u_{j}
				\right\Vert _{2} 
\end{equation}
holds.
\end{proposition}

\begin{remark}
\label{rem:brown}Let $T^{(j)}\psi=Pe_{k}u_{j}\overline{\psi}$ where $u_{j}\in
L^{2}\left(  {\bbC}\right)  $. Consider the form
\begin{equation}
\left\langle 1,e_{k}u_{0}T^{(1)}\cdots T^{(2n)}1\right\rangle
\label{eq:form.2n}%
\end{equation}
which defines a function of $k$. Integrating (\ref{eq:form.2n}) against a test
function $\widehat{\rho}$ in the $k$-variable and applying (\ref{ineq:brown})
shows that (\ref{eq:form.2n}) defines an $L^{2}$ function of $k$ with
\[
\left\Vert \left\langle 1,e_{k}u_{0}T^{(1)}\cdots T^{(2n)}1\right\rangle
\right\Vert _{2}\leq C_{n}\prod_{j=0}^{2n}\left\Vert u_{j}\right\Vert _{2}.
\]
For details we refer the reader to the proof of Theorem 2 in \cite{Brown:2001}
where a very similar estimate is proved.
\end{remark}

\section{An Oscillatory $\dbar$-Problem}
\label{sec:osc}

In this section we study the $\dbar$ problem \eqref{mu.dbar}. The main results used in \S \ref{sec:maps} are Lemmas \ref{lemma:mu.lip} and \ref{lemma:mu.expand}.  Because the problem \eqref{nu.dbar} has a nearly identical structure, the results of this section apply to the problem \eqref{nu.dbar} with typographical changes. Fix $p \in (2,\infty)$, $k \in \bbC$ and $u \in H^{1,1}(\bbC)$. It follows from Lemma \ref{solve} that a pair of functions $(\mu_1,\mu_2)$ with $\mu_1-1,\mu_2 \in L^p(\bbC)$ solves \eqref{mu.dbar.1}--\eqref{mu.dbar.2}  if and only if 
\begin{align}
\label{mu.dbar.int}
\mu_1 -1 &= T_k \mu_2 \\
\nonumber
\mu_2     &= T_k \mu _1
\end{align}
where $T_k$ is the antilinear operator
$$
\left[ T_k \psi \right](z) = \frac{1}{2} P \left[ e_k(\dotarg) u(\dotarg) \overline{\psi(\dotarg)} \right](z).
$$
We sometimes write $T_{k,u}$ for $T_k$ to emphasize its dependence on $u$. We will solve these integral equations and then check that \eqref{mu.dbar.3} holds for the solutions so constructed (see \eqref{mu.int.sol.asy}).

Formally, 
$\mu_1 = (I-T_k^2)^{-1} 1$. To prove and analyze this solution formula, we will need 
the following estimates which are easily deduced from \eqref{HLS}--\eqref{Holder}. Here $C_p$ (resp.\ $C_{p,q}$) represent numerical constants depending only on $p$ (resp.\  $p,q$). 
\begin{align}
\label{T1}
\norm{T_k}_{\calBbar(L^{p})} 							
	&\leq 	C_{p} \norm{u}_2 \\
\label{T2}
\norm{T_{k,u} - T_{k,u'}}_{\calBbar(L^p)}  			
	&\leq		C_p \norm{u-u'}_2\\
\label{T3}
\norm{T_k}_{\calBbar(L^p,L^\infty)} 					
	&\leq		C_{p} \left( \norm{u}_{2p/(p-1)} + \norm{u}_{p(p+2)/(p-2)} \right)\\
\label{T5}
\left| \left(T_k \psi\right)(z) - \left(T_k \psi\right)(z') \right| 
	&\leq C_{p_0,p} |z-z'|^{1-2/p_0} \norm{u}_{H^{1,1}} \norm{\psi}_p, \quad 2<p_0<p\\
\label{T6}
\norm{T_{k,u}-T_{k',u'}}_{\calBbar(L^p)} 
	&\leq C_p \left( 
								\norm{ \left(e_{k-k'}(\dotarg) - 1\right) u}_2 + 
								\norm{u-u'}_2 
					\right)\\
\label{T7}
\norm{ T_{k,u} - T_{k,u'}}_{\calBbar(L^p,L^\infty)} &\leq C_p \left(\norm{u-u'}_{2p/(p-1)}+\norm{u-u'}_{p(p+2)/(p-2)}\right)
\end{align}
\begin{multline}
\label{T7bis}
\norm{ T_{k,u} - T_{k',u}}_{\calBbar(L^p,L^\infty)}  \\
 		\leq C_p \left( \norm{(e_{k-k'}-1)u}_{2p/(p-1)} + \norm{(e_{k-k'}-1)u}_{p(p+2)/(p-2)} \right)
\end{multline}
In \eqref{T3}, we used \eqref{vekua} with $q=2p/(p+1)$ and $r=(p+2)/2$.
In \eqref{T5}, we used \eqref{Holder} together with $\norm{u\psi}_{p_0} \leq \norm{u}_s \norm{\psi}_p$ for $s^{-1}=p_0^{-1}-p^{-1}$. The estimate \eqref{T7} follows from \eqref{T3} and the linear dependence of $T_{k,u}$ on $u$.
The estimate \eqref{T7bis} follows from \eqref{T3} and the linear dependence of $T_{k,u}$ on $e_k u$.

We also have from \eqref{H11.Lp} and \eqref{HLS}  that
\begin{align}
\label{T8} 
\norm{T_k 1}_p &\leq C_p \norm{u}_{H^{1,1}} \\
\label{T9}
\norm{T_{k,u}1 - T_{k,u'} 1}_p &\leq C_p \norm{u-u'}_{H^{1,1}}
\end{align}
Using the inequality $|e^{i\theta}-1| \leq 2^{1-\alpha} |\theta|^\alpha$ for any $\alpha \in [0,1]$, \eqref{HLS}, and H\"{o}lder's inequality, we have
\begin{equation}
\label{T10}
\norm{T_k 1 - T_{k'} 1}_{p} \leq C_p |k-k'|^\alpha \norm{u}_{L^{2,1}}, \quad \alpha \in [0,1-2/p)
\end{equation}
while from \eqref{vekua} with $q=2p/(p+1)$, $r=(p+2)/2$, we have
\begin{equation}
\label{T11}
\norm{T_k 1 - T_{k'} 1}_\infty \leq C_p \left( \norm{(e_{k-k'} -1)u}_{2p/(p+1)} + \norm{(e_{k-k'}-1)u}_{(p+2)/2} \right).
\end{equation}

In what follows, it will be important to track uniformity of estimates for $u$ in bounded subsets of $H^{1,1}(\bbC)$. For given $M_0 >0$, we denote 
$$ B_0 = \left\{ u \in H^{1,1}(\bbC): \norm{u}_{H^{1,1}} \leq M_0 \right\}.$$
We denote by $C(M_0)$ a constant depending only on $M_0$.

We first construct the resolvent $(I-T_{k,u}^2)^{-1}$ for $(k,u) \in \bbC \times L^2(\bbC)$ using Fredholm theory. 

\begin{lemma}
\label{lemma:T.Fred}
For any $(k,u) \in \bbC \times L^2(\bbC)$ and $p \in (2,\infty)$, $(I-T_{k,u}^2)^{-1}$ exists as a bounded operator on $L^p$ and
the map
$$ (k,u) \mapsto (I-T_{k,u}^2)^{-1} $$
is continuous from $\bbC \times L^2$ into $\calB(L^p)$.
\end{lemma}

\begin{proof} First, we show that $T_k$ is a compact operator on $L^p$. By the norm-closure of compact operators, the estimate \eqref{T2}, and the density of $C_0^\infty(\bbC)$ in $L^2(\bbC)$, it suffices to show that $T_k$ is compact for $u \in C_0^\infty(\bbC)$. Let $p' \in (1,2)$ be the conjugate exponent to $p$. It suffices to show that the Banach space adjoint $T_k'=-\frac{1}{2} e_k u P$ is compact from $L^{p'}(\bbC)$ to itself.  Let $\Omega \subset \bbC$ be a bounded set with smooth boundary containing the support of $u$. If $f \in L^{p'}(\bbC)$ then $Pf \in L^{2p/(p-2)}(\bbC)$ by \eqref{HLS} while $\nabla Pf \in L^{p'}(\bbC)$ by Lemma \ref{lemma:S}. Thus
$$ \norm{u Pf}_{W^{1,p'}} \leq C\left(1+|\Omega|^{1/2} \right) \| f \|_{p'} $$
and compactness follows from the Rellich-Kondrakov Theorem.

Next, we recall the standard argument (see, for example, \cite[\S 7]{BC:1989}) that $\ker(I-T_k^2)$ is trivial. Suppose that $\psi \in L^p(\bbC)$ with $\psi=T_k^2 \psi$. Then
by Lemma \ref{solve}, the pair $(\psi,T\psi)$ is a weak solution of the system \eqref{mu.dbar.1}--\eqref{mu.dbar.2}. It follows that $\phi_+ = \psi+T\psi$ and $\phi_-= \psi-T\psi$ each solve the scalar problem $\dbar w = a \wbar$ with $a=\pm \frac{1}{2} e_k u \in L^2(\bbC)$ and $w \in L^p(\bbC)$. We now conclude from Lemma \ref{Liouville} that $\phi_+=\phi_-=0$ so $\psi=0$.

It now follows from the Fredholm alternative that $(I-T_{k,u}^2)^{-1}$ exists. To prove that the resolvent is continuous in $(k,u) \in \bbC \times L^2(\bbC)$, we appeal to \eqref{T1}, \eqref{T6}, the Dominated Convergence Theorem, and the second resolvent formula.
\end{proof}

For $u \in H^{1,1}(\bbC)$, the operator $T_k^2$ has small norm for large $|k|$.

\begin{lemma}
\label{lemma:Tk.decay}
Fix $p \in (2,\infty)$. For $u \in H^{1,1}(\bbC)$ and $|k| \geq 1$, the estimate
$$ \norm{ T_{k,u}^2 }_{\calB(L^p)} \leq C_p \norm{u}_{H^{1,1}}^2  \langle k \rangle^{-1} $$
holds. Moreover,
$$ \norm{ T_{k,u}^2 1 }_{p} \leq C_p \norm{u}_{H^{1,1}}^2 \langle k \rangle^{-1}. $$
\end{lemma}

\begin{proof}
From \eqref{ip1} with $f=u\psibar$ and $q=2p/(p+2)$ we have the estimate
\begin{equation}
\label{pre.Tk.decay}
 \norm{T_k \psi}_p 
 		\leq C_p 	\left\langle k \right\rangle^{-1}
 						\left( 
 							\norm{u}_2 \norm{\psi}_p + \norm{\dbar u}_2 \norm{\psi}_p + 
 							\norm{u}_p \norm {\dbar \psibar}_2 
 						\right)
\end{equation}
so that
$$
\norm{T_k^2 \psi}_p
		\leq	C_p		\left\langle k \right\rangle^{-1} 
							\left(
								\norm{u}_{2} \norm{T_k \psi}_{p} +
								\norm{\dbar u}_2 \norm{u}_2 \norm{\psi}_p +
								\norm{u}_p \norm{u}_{2p/(p-2)} \norm{\psi}_p
							\right).
$$
In the second term we used \eqref{T1}, and in the third term,  we used $\norm{\dbar \overline{T_k \psi}}_2=\norm{\dee \overline{T_k \psi}}_2=\norm{u\psi}_2$ (the second step follows from the unitarity of the Beurling transform; see Lemma \ref{lemma:S}).  Using \eqref{H11.Lp} and \eqref{T1}, we obtain the first estimate.

To obtain the second estimate, we use \eqref{pre.Tk.decay} with $\psi=T_k 1$ together with \eqref{H11.Lp} and \eqref{HLS}.
\end{proof}

From Lemma \ref{lemma:T.Fred} and Lemma \ref{lemma:Tk.decay}, we obtain the following  uniform estimate on the resolvent.

\begin{lemma}
\label{lemma:res.bd}
Fix $M_0>0$ and $p > 2$. The estimate
$$
\sup 
	\left\{ \norm{(I-T_{k,u}^2)^{-1}}_{\calB(L^p)} :  k \in \bbC, \,   u \in B_0 \right\}
	\leq C(M_0,p) 
$$
holds. 
\end{lemma}

\begin{proof}
By Lemma \ref{lemma:Tk.decay}, given $M_0$, we can find $R_0$ so that 
$ \norm{(I-T_{k,u}^2)^{-1}}_{\calB(L^p)} \leq 2 $ for all $(k,u)$ with $|k| > R_0$ and $\norm{u}_{H^{1,1}} \leq M_0$. On the other hand, the set 
$$\{ (k,u): |k| \leq R_0, \norm{u}_{H^{1,1}} \leq M_0 \}$$ 
is bounded in $\bbC \times H^{1,1}$, hence precompact in $\bbC \times L^2$ by Lemma \ref{lemma:H11.Lp}. The image of this set under the map
$(k,u) \mapsto \left(I-T_{k,u}^2\right)^{-1}$ is therefore a bounded set in $\calB(L^p)$ by the continuity asserted in Lemma \ref{lemma:T.Fred}.
\end{proof}

\begin{lemma}
\label{lemma:mu.sol}
Let $u \in H^{1,1}(\bbC)$. For each $k \in \bbC$ and any $p\in (2,\infty)$, the functions
\begin{equation}
\label{mu.int.sol}
\mu_1 := 1 + (I-T_k^2)^{-1} T_k^2 1, \qquad \mu_2 := T_k \mu_1 
\end{equation}
are the unique solutions of \eqref{mu.dbar.int} with  $\mu_1 -1, \mu_2 \in C^\alpha \cap L^p$. For these solutions, 
\begin{equation}
\label{mu.int.sol.asy}
\lim_{|z|\rarr \infty} (\mu_1,\mu_2) = (1,0).
\end{equation} 
Moreover, for any $p \in (2,\infty)$, any $\alpha \in [0,1-1/(2p))$, any $M_0>0$ and any $u \in B_0$, the following estimates hold:
\begin{align}
\label{mu.int.bd}
\sup_{k \in \bbC} \left( \norm{\mu_1-1}_{p} + \norm{\mu_2}_{p} \right) \leq
	C(M_0,p) \\
\label{mu.int.holder.z}
\sup_{k \in \bbC}  \left( \norm{\mu_1}_{C^\alpha(\bbC)} + \norm{\mu_2}_{C^\alpha(\bbC)} \right) \leq
	C(M_0,\alpha) 
\end{align}
\end{lemma}

\begin{proof}
Let $p>p_0>2$. It follows from \eqref{T3} and \eqref{T5} that for $u \in H^{1,1}$, $T_k$ maps $L^p(\bbC)$ into $C^{1-2/p_0}(\bbC)$ with bound uniform in $k\in \bbC$ and $u$ in bounded subsets of $H^{1,1}(\bbC)$. 
From \eqref{T8}, we have $T_k 1 \in L^p$  with norm bounded uniformly in $k \in \bbC$ and $u$ with $\norm{u}_{H^{1,1}} \leq M_0$. It follows from this fact and Lemma \ref{lemma:T.Fred} that $(I-T_k^2)^{-1} T_k^2 1=T_k (I-T_k^2)^{-1} T_k 1 \in L^p(\bbC) \cap C^{1-2/p_0}(\bbC)$ for $p_0>p>2$. Hence, the functions given by \eqref{mu.int.sol} solve \eqref{mu.dbar.int} with $\mu_1-1 \in L^p \cap C^{1-2/p_0}$. The assertion about limiting behavior follows from \eqref{vanish}.

The estimate \eqref{mu.int.bd} follows from \eqref{T1}, \eqref{T8}, and Lemma \ref{lemma:res.bd}. The estimate \eqref{mu.int.holder.z} follows from the uniform estimate on $\norm{T_k}_{\calBbar(L^p,C^{1-2/p_0})}$.
\end{proof}

Next, we study the $k$-dependence of the solutions \eqref{mu.int.sol}. For brevity we write $\mu=(\mu_1,\mu_2)$. 

\begin{lemma}
Fix $p >2$,  $M_0>0$ and $u \in B_0$. Then
\begin{equation}
\label{mu.holder.k.p}
\norm{\mu(\dotarg,k)-\mu(\dotarg,k')}_{L^p} \leq C(M_0,\alpha) |k-k'|^\alpha
\end{equation}
for any $\alpha \in [0,1-2/p)$, and
\begin{equation}
\label{mu.holder.k}
\sup_{z \in \bbC} \left| {\mu(z,k)-\mu(z,k')} \right| \leq C(M_0,\alpha) F(|k-k'|)
\end{equation}
where $F(0)=0$, $F$ is continuous, and $F$ depends only on $u$ and $p$. 
Finally, for each fixed $z \in \bbC$, 
\begin{equation}
\label{mu.k.asy}
\lim_{|k| \rarr \infty} (\mu_1(z,k), \mu_2(z,k))=(1,0).
\end{equation}
\end{lemma}

\begin{proof}
From \eqref{T6} and the inequality $\left| e_k(z) -1 \right| \leq 2^{1-\alpha} |k|^\alpha |z|^\alpha$, we easily see that for any $p \in (2,\infty)$ and $\alpha \in [0,1]$, the estimate 
\begin{equation}
\label{Tk.holder}
\norm{T_k - T_{k'}}_{\calBbar(L^p)} \leq C_{p,\alpha} \norm{u}_{H^{1,1}} |k-k'|^\alpha
\end{equation}
holds. From this estimate, Lemma \ref{lemma:res.bd}, and the second resolvent formula, we conclude that
$\norm{(I-T_k^2)^{-1}-(I-T_{k'}^2)^{-1}}_{\calB(L^p)} \leq C(M_0,p,\alpha) |k-k'|^\alpha$ for any $\alpha \in [0,1]$. From \eqref{T10}, \eqref{mu.int.sol}, and \eqref{Tk.holder} again, we conclude that
$$\norm{\mu_1(\dotarg,k)-\mu_1(\dotarg,k')}_{L^p} \leq C(M_0,p,\alpha) |k-k'|^\alpha, \quad \alpha \in [0,1-2/p).$$
A similar estimate holds for $\mu_2$ by the formula $\mu_2 = T_k \mu_1$, the estimates \eqref{T1} and \eqref{T10}, the continuity estimate on $\mu_1$, and \eqref{Tk.holder}.  This proves \eqref{mu.holder.k.p}.

Using \eqref{mu.holder.k.p}, the estimates \eqref{T1}, \eqref{T3},\eqref{T6}, \eqref{T7bis}, \eqref{T11}, \eqref{mu.int.bd}, and the identity
$$\mu_1(z,k)-\mu_1(z,k') = \left( T_{k}^2 \mu_1\right)(z,k)  - \left(T_{k'}^2 \mu_1\right)(z,k'),$$
we conclude that \eqref{mu.holder.k} holds for $\mu_1$ with
\begin{align*}
F(k-k') 	&= 	\quad \norm{(e_{k-k'}-1)u}_{2p/(p-1)} + \norm{(e_{k-k'}-1)u}_{(p+2)/p} \\
			&\quad	+	\norm{(e_{k-k'}-1)u}_{2p/(p-1)} + \norm{(e_{k-k'}-1)u}_{p(p+2)/(p-2)}.
\end{align*}
%Applying the estimates \eqref{mu.int.bd}, \eqref{Tk.holder} and \eqref{T1}, we obtain a similar estimate for $\mu_2$.
To estimate $\mu_2$ we write $\mu_2=T_k \mu_1$ and use \eqref{T3}, \eqref{T7bis}, \eqref{T11},  \eqref{mu.int.bd},
and \eqref{mu.holder.k.p} to obtain an estimate with the same $F$ as above.

From \eqref{mu.int.holder.z} and \eqref{mu.holder.k}, it follows that the solutions \eqref{mu.int.sol} are jointly  continuous in $(z,k)$ so that, in particular, point evaluations make sense. If we can show that $\lim_{|k| \rarr \infty} T_k 1 = 0$ in $L^p(\bbC)$, \eqref{mu.k.asy} will follow from \eqref{mu.int.sol}, Lemma \ref{lemma:res.bd}, and the uniform estimate \eqref{T3}. By \eqref{T9} and a density argument, it suffices to show that $\lim_{|k| \rarr \infty} T_k 1 = 0$ for $u \in C_0^\infty(\bbC)$. This is an immediate consequence of \eqref{ip0}.
\end{proof}

Next, we prove Lipschitz continuity of $\mu$ as a function of $u$. We write $\mu(z,k,u)$ to emphasize the dependence of $\mu$ on $u$. 

\begin{lemma}
\label{lemma:mu.lip}
Fix $M_0 >0$ and $p \in (2,\infty)$, and suppose that $u, u' \in B_0$.  Then
\begin{align}
\label{mu.lip.Lp}
\sup_{k \in \bbC} \norm{\mu(\dotarg,k,u) - \mu(\dotarg,k,u')}_p &\leq C(M_0) \norm{u-u'}_{H^{1,1}}\\
\label{mu.lip.C0}
\sup_{(z,k) \in \bbC \times \bbC} \left| \mu(z,k,u) - \mu(z,k,u') \right| &\leq C(M_0) \norm{u-u'}_{H^{1,1}}
\end{align}
\end{lemma}

\begin{proof}
If \eqref{mu.lip.Lp} holds, we can use \eqref{T1}, \eqref{T2}, \eqref{T3}, \eqref{T7}, and the identity $(\mu_1-1,\mu_2)=(T_k^2 \mu_1,T_k \mu_1)$ to conclude that \eqref{mu.lip.C0} holds. 

If \eqref{mu.lip.Lp} holds with $\mu$ replaced by $\mu_1$, then the same estimate for $\mu_2$ follows from the formula  $\mu_2 = T_k 1 + T_k( \mu_1-1)$ and \eqref{T1}, \eqref{T2}, and \eqref{T9}. 

It remains to prove \eqref{mu.lip.Lp} for $\mu$ replaced by $\mu_1$. By the second resolvent formula,  \eqref{T1}, \eqref{T2}, and Lemma \ref{lemma:res.bd}, for any $u$, $u'$ in $B$,
\begin{equation}
\label{res.lip}
\norm{ \left(I-T_{k,u}^2\right)^{-1} -   \left(I-T_{k,u'}^2\right)^{-1} }_{\calB(L^p)} 
	\leq C(M_0,p) \norm{u-u'}_2.
\end{equation}
Using the identity
$$ 
\mu_1(\dotarg,k,u) - \mu_1(\dotarg,k,u') 
	= \left[ T_{k,u} (I-T_{k,u}^2)^{-1} T_{k,u} 1 - T_{k,u'}(I-T_{k,u'}^2)^{-1} T_{k,u'} 1 \right],
$$
%the Lipschitz estimates \eqref{T7}, \eqref{T9} and  \eqref{res.lip}, the uniform bounds \eqref{T3}, and Lemma \ref{lemma:res.bd}, 
the estimates \eqref{T1}, \eqref{T2}, \eqref{T8}, \eqref{T9}, the uniform estimate from Lemma \ref{lemma:res.bd}, and the Lipschitz estimate \eqref{res.lip},
we conclude that $\mu_1$ satisfies the $L^p$ Lipschitz estimate.
\end{proof}

We now turn to the scattering map \eqref{R}. If $u \in H^{1,1}(\bbC)$ we may define $r=\calR u$ by
\begin{equation}
\label{r.split}
r(k) 	= \frac{1}{\pi}\int e_k(z) u(z) \, dA(z)
		 + \frac{1}{\pi} \int e_k(z) u(z) \left(\overline{\mu_1(z,k)}-1\right) \, dA(z).
\end{equation}
The first term is a Fourier transform and is well-defined as an element of $H^{1,1}(\bbC)$. The second integral defines a bounded continuous function of $k$ by \eqref{mu.holder.k.p} since $u \in L^{p'}$. 
It follows from \eqref{r.split} that if  $r=\calR u$,  $r'=\calR u'$, and $\norm{u}_{H^{1,1}}, \norm{u'}_{H^{1,1}} \leq M_0$, then
\begin{align*}
r-r' 	&	= \frac{1}{\pi} \int e_{k} (u-u') \\
		&\quad	+ \frac{1}{\pi} 
						\int 
								e_k \left(
											  u \left[\overline{\mu_1(z,k;u)}-1\right] 
											- u' \left[\overline{\mu_1(z,k;u')}-1\right] 
										\right) \\
		&=I_1(k)+I_2(k)
\end{align*}
where for any $p \in (2,\infty)$
\begin{equation}
\label{r.lip.0}
\norm{I_1}_p \leq C_p \| u - u' \|_{p'}, \qquad \norm{I_2}_{\infty} \leq C(M_0) \norm{ u - u'}_{H^{1,1}}.
\end{equation}
The first estimate follows from the Hausdorff-Young inequality. In the second estimate, we used \eqref{mu.lip.Lp} and \eqref{mu.int.bd}.

\begin{remark}
\label{rem:r}
If $u \in H^{1,1}(\bbC) \cap L^1(\bbC)$ and $r=\calR u$, we may compute
$$ r(k) = \int e_k(z) u(z) \mu_1(z,k) \, dA(z) $$
where the integral is absolutely convergent by \eqref{mu.int.holder.z}. By the Dominated Convergence Theorem and \eqref{mu.holder.k.p}, this expression defines a continuous function of $k$.  

We claim that, moreover, $r \in C_0(\bbC)$, the continuous functions vanishing at infinity. To see this we use \eqref{r.split} and note that the first term vanishes as $|k| \rarr \infty$ by the Riemann-Lebesgue lemma, while the second term vanishes as $|k| \rarr \infty$ by \eqref{mu.k.asy} and dominated convergence.
\end{remark}

We now give a self-contained proof of  the standard result (see \cite[\S 3.2]{BC:1989}, the formal argument in \cite[I, \S 1]{Sung:1994}, and justification in \cite[II, \S 4]{Sung:1994})  that the functions $(\mu_1,\mu_2)$ are determined by the $\dbar$-data $r$. More precisely, we will show that the functions
\begin{equation}
\label{nu.mu}
\nu_1 = \mu_1, \quad \nu_2 = e_k \overline{\mu_2}
\end{equation}
solve the $\dbar_k$ problem \eqref{nu.dbar}. We will prove this by direct differentiation of the solution formulas \eqref{mu.int.sol} in $k$. 

To do so, we will need the following well-known lemma which shows that the `analyticity defect' of the operator $T_k^2$ is a rank-one operator. We give a proof for completeness.

\begin{lemma}
The identity 
\begin{equation}
\label{BC.Miracle}
\left( \dbar_k T_k^2 \right) = \frac{1}{2} (T_k 1) \calF^{-1}( \ubar \dotarg)
\end{equation}
holds as a derivative in $\calB(L^p)$ operator norm, where $\calF^{-1}$ is given by \eqref{Finv}. 
\end{lemma}

\begin{proof}
This identity is formally obvious but we need an explicit estimate to prove differentiability in operator norm. Write $z=x_1+ix_2$ and $k=k_1+ik_2$ so that $e_k(z)=\exp(-2i(k_1 x_2+k_2 x_1))$ and $\dbar_k = (1/2)\left(\dee_{k_1}+i\dee_{k_2}\right)$. We claim that
\begin{multline}
\label{BC.k1}
\left( \frac{\dee}{\dee k_1} T_k^2 f \right)(z) \\
	= -\frac{i}{2\pi^2} \int \frac{1}{z-z'} e_k(z'-z'') \frac{x_2'-x_2''}{\zbar'-\zbar''} u(z') u(z'') f(z'')\, dA(z'') \, dA(z')
\end{multline}
\begin{multline}
\label{BC.k2}
\left( \frac{\dee}{\dee k_2} T_k^2 f \right)(z) \\
	= -\frac{i}{2\pi^2} \int \frac{1}{z-z'} e_k(z'-z'') \frac{x_1'-x_1''}{\zbar'-\zbar''} u(z') u(z'') f(z'') \, dA(z'') \, dA(z')
\end{multline}
from which \eqref{BC.Miracle} follows. We will prove \eqref{BC.k1} since the proof of \eqref{BC.k2} is similar. Using the estimate $$\left|e^{iht}-1-iht \right| \leq 2^{1-\theta} |h|^{1+\theta} |t|^{1+\theta}$$ (for $\theta \in (0,1)$ to be chosen), denoting $T_k^2 f$ by $F(k_1)$, and denoting the right-hand side of \eqref{BC.k1} by $F'(k_1)$, we can estimate 
$\left|h^{-1} \left(F(k_1+h)-F(k_1) \right) - F'(k_1)\right|$ by $|h|^\theta$ times 
\begin{equation}
\label{BC.int.bd}
\int \frac{\left|x_{2}'-x_{2}'' \right|^\theta}{|z-z'|}  \,  |u(z')| \, |u(z'')| \, |f(z'')| \, dA(z'') \, dA(z'). 
\end{equation}
To prove norm differentiability, it suffices to bound \eqref{BC.int.bd} as an $L^p$ function of $z$ uniformly in $f$ with $\norm{f}_p \leq 1$. The $L^p$ norm of the expression \eqref{BC.int.bd} is bounded by $2^{-\theta}$ times the sum of the $L^p$ norms of the functions
\begin{align*}
I_1(z) &= \int \frac{1}{|z-z'|} |z'|^\theta \,  |u(z')| \, |u(z'')| \,  |f(z'')| \, dA(z'') \, dA(z')\\
I_2(z) &= \int \frac{1}{|z-z'|} |z''|^\theta \,  |u(z')| \, |u(z'')| \, |f(z'')| \, dA(z'') \, dA(z')
\end{align*}
By H\"{o}lder's inequality and \eqref{HLS},
\begin{align*}
\norm{I_1}_p &\leq \norm{u}_{p'} \norm{(\dotarg)^\theta u(\dotarg)}_{2p/(p+2)}\\
\norm{I_2}_p &\leq \norm{u}_{2p/(p+2)} \norm{ (\dotarg)^\theta u(\dotarg)}_{p'}
\end{align*}
so it suffices to choose $\theta$ so the weighted norms of $u$ are bounded for $u \in H^{1,1}$. A short calculation shows that the norm $\norm{(\dotarg)^\theta u(\dotarg)}_s$ is bounded by constants times $\norm{\left\langle \dotarg \right\rangle u(\dotarg)}_2$ provided $0 < \theta < 2- 2/s$. Choosing any $\theta$ with $0<\theta<\min(1-2/p,2/p)$ gives the desired bound.
\end{proof}

First, we consider $u \in \calS(\bbC)$. 

\begin{lemma}
\label{lemma:mu.dbar.k.1}
Let $u \in \calS(\bbC)$, and let $(\mu_1,\mu_2)$ be given by \eqref{mu.int.sol}. Then, for each $z\in \bbC$,  $\nu_1(z,\dotarg),\nu_2(z,\dotarg)$ defined by \eqref{nu.mu} are strong solutions of the system \eqref{nu.dbar}.
\end{lemma}

\begin{proof}
The asymptotic condition \eqref{nu.dbar.3} is an immediate consequence of \eqref{mu.k.asy} and the definition of $(\nu_1,\nu_2)$. 
To show that $(\nu_1,\nu_2)$ satisfy \eqref{nu.dbar.1}--\eqref{nu.dbar.2}, it suffices to show that
\begin{equation}
\label{pre.mu.dbar.k}
\dbar_k \mu_1 = \frac{1}{2} \rbar \mu_2, \qquad 
\left(\dee_k + z\right) \mu_2 =\frac{1}{2}r \mu_1.
\end{equation}
We will prove these identities by differentiating the solution formulas \eqref{mu.int.sol} with respect to $k$ and using \eqref{BC.Miracle}.

For $u \in \calS(\bbC)$ it is easy to see that
$$ \dbar_k T_k^2 1 =(\calF^{-1} \ubar)(k) (T_k 1) $$
where the Fourier transform defines a continuous function of $k$ since $u \in \calS(\bbC)$. Using the operator identity
$$ \dbar_k (I-T_k^2)^{-1} = (I-T_k^2)^{-1} \left( \dbar_k T_k^2 \right) (I-T_k^2)^{-1} $$ 
together with the formula
$ \mu_1 - 1 = (I-T_k^2)^{-1} T_k^2 1 $
and \eqref{BC.Miracle}, we compute
\begin{align*}
\dbar_k \mu_1 	&= \dbar_k \left( (I-T_k^2)^{-1} T_k^2 1 \right) \\
						&= \left[(I-T_k^2)^{-1} T_k 1 \right] (\calF^{-1}(\ubar (\mu_1 - 1)) 
							+ \left[(I-T_k^2)^{-1} T_k 1 \right] \calF^{-1}(\ubar)\\
						&= \mu_2 \calF^{-1}(\ubar \mu_1)\\
						&= \rbar \mu_2
\end{align*}
To compute $(\dee+k)\mu_2$ we will use the identity
$(\dee_k + z) T_k f = \calF\left( u \overline{f}\right)+ T_k  \left( \dbar_k f\right)$,  which holds pointwise if $u \in \calS(\bbC)$, $f(\dotarg,k) \in C(\bbC)$, and $(\dbar_k f)(\dotarg,k) \in C(\bbC)$ with bounds uniform in $k$. We then compute
\begin{align*}
\left(\dee_k + z\right) \mu_2 
	&= \left(\dee_k + z \right) T_k \mu_1 \\
	&=  \calF(u \overline{\mu_1}) + T_k \left(\dee_k \mu_1 \right) \\
	&=  r+ r T_k \mu_2\\
	&= r \mu_1
\end{align*}
where in the last step we used $T_k \mu_2 = T_k^2 \mu_1 = \mu_1 -1$.
\end{proof}

Next, we use Lemma \ref{lemma:mu.lip} to extend the result to $u \in H^{1,1}$.

\begin{lemma}
\label{lemma:mu.dbar.k}
Let $u \in H^{1,1}(\bbC)$ and let $(\mu_1,\mu_2)$ be given by \eqref{mu.int.sol}. Then, for each $z\in \bbC$,  $\nu_1(z,\dotarg),\nu_2(z,\dotarg)$ defined by \eqref{nu.mu} are weak solutions of the system \eqref{nu.dbar}.
\end{lemma}

\begin{proof}
It suffices to show that for each $\varphi \in C_0^\infty(\bbC)$ and each fixed $z \in \bbC$,
\begin{align}
\label{weak.1}
\int \left(-\dbar_k \varphi \right)\mu_1(z,k) \, dA(k) 
	&= \frac{1}{2} \int \varphi(k) \overline{r(k)} {\mu_2(z,k)} \, dA(k)\,\\
\label{weak.2}
\int \left( -\dee + z \right) \varphi(k) \mu_2(z,k) \, dA(k)
	&= \frac{1}{2} \int \varphi(k) r \mu_1(z,k) \, dA(k).
\end{align}
Let $\{ u_n \}_{n=1}^\infty$ be a sequence from $C_0^\infty(\bbC)$ converging in $H^{1,1}(\bbC)$ to $u$, and denote by $\mu_{1,n}, \mu_{2,n}$ the corresponding solutions given by \eqref{mu.int.sol}.  Finally, let
$ r_n = \pi^{-1} \int e_k u_n \overline{\mu_{1,n}}$. 
By Lemma \ref{lemma:mu.dbar.k.1} and an integration by parts, the identities
\begin{align}
\label{weak.1.n}
\int \left(-\dbar_k \varphi \right)\mu_{1,n}(z,k) \, dA(k) 
	&= \frac{1}{2} \int \varphi(k) \overline{r_n(k)} {\mu_{2,n}(z,k)} \, dA(k)\\
\label{weak.2.n}
\int \left( \left(-\dee + z \right) \varphi(k) \right) \mu_{2,n}(z,k) \, dA(k)
	&= \frac{1}{2} \int \varphi(k) r_n \mu_{1,n}(z,k) \, dA(k)
\end{align}
hold. We will prove \eqref{weak.1}--\eqref{weak.2} by taking limits in \eqref{weak.1.n}--\eqref{weak.2.n} as $n \rarr \infty$. We give the proof for \eqref{weak.1} since the other is similar. The left-hand side of \eqref{weak.1.n} converges to the left-hand side of \eqref{weak.1} as $n \rarr \infty$ by \eqref{mu.lip.C0}. 
To show convergence of the right-hand side, we estimate 
\begin{align*}
\left| \int  \varphi(k) \left( \overline{r_n(k)} \, {\mu_{2,n}} - \overline{r(k)}\,  {\mu_2(z,k)} \right) \, dA(k) \right|
	&\leq C(M_0) \norm{u_n - u}_{H^{1,1}}
\end{align*}
where we used uniform bounds \eqref{mu.int.bd}, \eqref{mu.int.holder.z} together with Lipschitz estimates    \eqref{mu.lip.Lp}, \eqref{mu.lip.C0}, and \eqref{r.lip.0}.
\end{proof}

We now briefly discuss the $\dbar$-problem \eqref{nu.dbar} and prove:

\begin{lemma}
\label{lemma:IR}
For any $r \in \calS(\bbC)$, the relation \eqref{IR} holds.
\end{lemma}

\begin{proof}
 Let $S_z$ be the antilinear operator
$$ \left[ S_z \psi \right] (k) = \frac{1}{2} P_k \left[ e_{(\dotarg)}(z) \overline{r(\dotarg)} \overline{\psi(\dotarg)} \right](k) $$
where $P_k$ is the Cauchy transform acting on the $k$ variable. Write $S_z = S_{z,r}$ to emphasize the dependence of $S_z$ on $r$. Observe that, as operators on $L^p(\bbC)$, we have 
\begin{equation}
\label{TS}
\left[S_{z,r} f \right](k) = \left[ T_{z,\rbar}f \right](k)
\end{equation}
Formally, \eqref{nu.dbar} is solved by
\begin{equation}
\label{nu.int.sol}
\nu_1 := 1+ \left(I-S_z^2\right)^{-1} S_z^2 1, \qquad \nu_2 := S_z \nu_1
\end{equation}
(compare \eqref{mu.int.sol}). Tracing through the proofs of Lemmas \ref{lemma:T.Fred}--\ref{lemma:mu.sol} one can easily prove that these formulas give the unique solution to \eqref{nu.dbar}. The uniqueness of solutions to the $\dbar$-problems for $\mu$ and $\nu$ and the identity $e_k(z) = e_z(k)$ easily imply the identity
\begin{equation}
\label{numu}
\nu_1(z,k,r) = \mu_1(k,z,\rbar).
\end{equation}

One may then compute, for $r \in \calS(\bbC)$,
\begin{align*}
(C \compose \calR \compose C)(r) 
	&=  C \compose \calR(\rbar) \\
	&= C \left( \frac{1}{\pi} \int e_z(k) \overline{r(k)} \overline{\mu_1(k,z;\rbar)} \, dA(k) \right)\\
	&= \frac{1}{z} \int e_{-k}(z) r(k) \nu_1(z,k,r) \, dA(k) \\
	&= \calI(r)
\end{align*}
where we used $e_k(z)=e_z(k)$ and, in the third line, we used \eqref{numu}.
\end{proof}

Finally, we obtain expansions for the solution $\mu$ which will facilitate a finer analysis of the scattering map.

\begin{lemma}
\label{lemma:mu.expand}
Fix $M_0>0$ and suppose that $u \in B_0$. Then for any positive integer $N$,
\begin{align*}
\mu_1(z,k) &= 1+ \sum_{j=1}^N T^{2j} 1 + R_{1,N}(z,k;u)\\
\mu_2(z,k) &= \sum_{j=0}^N T^{2j+1} 1 + R_{2,N}(z,k;u)
\end{align*}
where for any $p \in (2,\infty)$
\begin{align}
\label{R1.est}
\norm{ R_{1,N}(\dotarg,z;u) }_p &\leq C(p,M_0) \left\langle k \right\rangle^{-N-1} \\
\label{R2.est}
\norm{ R_{2,N}(\dotarg,z;u) }_p &\leq C(p,M_0) \left\langle k \right\rangle^{-N-1}
\end{align}
Moreover for any $p \in (2,\infty)$ and $u, u' \in B_0$, the estimates
\begin{align}
\label{R1.lip}
\sup_{k \in \bbC} \left\langle k \right\rangle^N \norm{R_{1,N}(\dotarg,k,u)-R_{1,N}(\dotarg,k,u')}_p
	&\leq C(M_0,p) \norm{u-u'}_{H^{1,1}}\\
\label{R2.lip}
\sup_{k \in \bbC} \left\langle k \right\rangle^N \norm{R_{2,N}(\dotarg,k,u)-R_{2,N}(\dotarg,k,u')}_p
	&\leq C(M_0,p) \norm{u-u'}_{H^{1,1}}
\end{align}
hold.
\end{lemma}

\begin{proof}
By iterating the integral equations \eqref{mu.dbar.int}, we see that
\begin{align*}
\mu_1 &= 1 + \sum_{j=1}^N T_k^{2j} 1 + T_k^{2N+2} \mu_1 \\
\mu_2 &= \sum_{j=0}^N T_k^{2j+1} 1 + T_k^{2N+3}\mu_2
\end{align*}
Thus $R_{1,N} = T_k^{2N+2} \mu_1$, $R_{2,N}=T_l^{2N+3} \mu_2$. The norm estimates \eqref{R1.est}--\eqref{R2.est} follow from \eqref{T1}, \eqref{T8},  Lemma \ref{lemma:Tk.decay}, and \eqref{mu.int.bd}. The Lipschitz estimates \eqref{R1.lip}--\eqref{R2.lip} follow from Lemma \ref{lemma:Tk.decay} and \eqref{mu.lip.Lp}.
\end{proof}

\section{Direct and Inverse Scattering Transforms}
\label{sec:maps}

In this section we study the direct and inverse maps $\calR$ and $\calI$ defined respectively by \eqref{mu.dbar}--\eqref{R} and \eqref{nu.dbar}--\eqref{I}. As in \S \ref{sec:osc}, for given $M_0>0$,  $B_0$ denotes the ball of radius $M_0$ in $H^{1,1}(\bbC)$.

First, we will prove:
\begin{proposition}
\label{prop:R}
The map $\calR$ defined initially on $\calS(\bbC)$ by \eqref{mu.dbar} and \eqref{R} extends to $H^{1,1}(\bbC)$. Moreover, for any $M_0>0$, $u, u' \in B_0$, we have $\calR u, \calR u' \in H^{1,1}(\bbC)$ and 
$$ \norm{\calR u - \calR u'}_{H^{1,1}} \leq C(M_0) \norm{u-u'}_{H^{1,1}}. $$
\end{proposition}

\begin{remark}
\label{rem:R}
The proof  of Proposition \ref{prop:R} shows that for $u \in H^{1,1}(\bbC)$, the scattering transform can be computed as
$$ 
(\calR u)(k) = 
		\calF (u)(k) + \frac{1}{\pi} \int e_k(\zeta) u(\zeta) \left( \overline{\mu_1(\zeta,k)}-1 \right) \, dA(\zeta)
$$
where the second right-hand term is an absolutely convergent integral for each $k$.
\end{remark}

We prove Proposition \ref{prop:R} in several steps. 

\begin{lemma}
\label{lemma:R1}
Fix $M_0>0$. For $u,u' \in B_0$, $\calR u$ and $\calR u'$ belong to $L^2(\bbC)$ and the estimate
$$ \norm{\calR u - \calR u'}_2 \leq C(M_0) \norm{u-u'}_{H^{1,1}} $$
holds.
\end{lemma}

\begin{proof}
We use Lemma \ref{lemma:mu.expand} with $N=2$. Substituting the expansion for $\mu_1$ into the integral formula \eqref{r.split} we see that
\begin{align*}
\calR u 
	&= \frac{1}{\pi} \int e_k u(z) \, dA(z)  + \frac{1}{\pi} \int e_k(z) u(z) \left(T_k^2 1 + T_k^4 1 \right) \, dA(z)\\
	&\quad + \frac{1}{\pi} \int e_k u(z) R_{1,2}(z,k)\, dA(z).
\end{align*}
The first term is a Fourier transform which is Lipschitz continuous as a map from $H^{1,1}$ to $L^2$. The second two terms are multilinear forms in $u$ and define $L^2$ functions of $k$ by Remark \ref{rem:brown}. Lipschitz continuity follows from multilinearity. Since $R_{1,2}(\cdot,k)$ has $L^p$ norm of order $\left\langle k \right\rangle^{-2}$, it follows from H\"{o}lder's inequality and \eqref{R1.est} that the last right-hand term defines a function in $L^2$, Lipschitz continuous in $u$ by \eqref{R1.lip}.
\end{proof}

Now we extend the Lipschitz estimates to the weighted space $L^{2,1}(\bbC)$. Initially we compute for $u \in \calS(\bbC)$ to justify the integrations by parts that occur.

\begin{lemma}
\label{lemma:R2}
Fix $M_0>0$. For $u, u' \in B_0$, $\calR u $ and $\calR u'$ belong to $L^{2,1}(\bbC)$ and
$$ \norm{\calR u - \calR u'}_{L^{2,1}} \leq C(M_0) \norm{u-u'}_{H^{1,1}}. $$
\end{lemma}

\begin{proof}
By Lemma \ref{lemma:R1}, it suffices to show that the map $u \mapsto (\dotarg) r(\dotarg)$ is well-defined and Lipschitz continuous from $H^{1,1}$ to $L^2$. We will prove Lipschitz continuity on the dense subset $\calS(\bbC)$ and extend by continuity to $H^{1,1}(\bbC)$. 

Using the trivial identity $\dee_z(e_k)=-ke_k$ and integrating by parts in \eqref{r.split}, we see that
$ k r(k) = \calF(\dee_z u)+ I_1+I_2$ where
\begin{align*}
I_1 &= \frac{1}{\pi} 
			\int e_k(\zeta) 
					\left(\dee_\zeta u\right)(\zeta) 
					\left(\overline{\mu_1(\zeta,k)}-1\right) 
				\, dA(\zeta)\\
I_2 &= \frac{1}{2\pi} \int |u(\zeta)|^2 \mu_2(\zeta,k) \, dA(\zeta)
\end{align*}
where in the second line we used \eqref{mu.dbar.1}.

To analyze $I_1$, let $\eta \in C_0^\infty(\bbC)$ with $\eta(z)=1$ for $|z| \leq 1$ and $\eta(z)=0$ for $|z| \geq 2$. Then $I_1=I_{11}+I_{12}$ where
\begin{align*}
I_{11}  &  =\frac{1}{\pi}
					\int e_{k}(\zeta)
							\eta(\zeta) \, \left(  \partial_{\zeta}u\right)  (\zeta)  \,
										\left[\overline{\mu_{1}(\zeta,k)}-1 \right] ~dA(\zeta),\\
I_{12}  &  =\frac{1}{\pi}
					\int e_{k}(\zeta) 
							\left(  1-\eta(\zeta)\right)  \, \left(  \partial_{\zeta}u\right)(\zeta) \,
										\left[ \overline{\mu_{1}(\zeta,k)}-1\right] ~dA(\zeta).
\end{align*}
In $I_{11}$, the function $\eta \dee_\zeta u$ belongs to $L^{p'}$ for any $p \in (2,\infty)$, so we can show that $I_{11}$ has the required continuity properties by mimicking the proof of Lemma \ref{lemma:R1} with $u$ replaced by $\eta \dee_\zeta u$. In $I_{12}$, substitute
\begin{align}
\label{mu1.exp}
\mu_1(\zeta,k) -1 
	&= \frac{1}{2\pi \zeta} \int e_k(\zeta') u(\zeta') \,
				\overline{\mu_2(\zeta',k)} \, dA(\zeta')\\
\nonumber	
	& + \frac{1}{2\pi\zeta} \int \frac{e_k(\zeta')}{\zeta-\zeta'} \,  \zeta' u(\zeta') \,
				\overline{\mu_2(\zeta',k)} \, dA(\zeta')
\end{align}
Inserting the second right-hand term of \eqref{mu1.exp} in $I_{12}$ leads to an integral that can be analyzed along the same lines as $I_1$ since $(1-\eta(\zeta))\zetabar^{-1} (\dee_\zeta u)(\zeta)$ belongs to $L^{p'}$ for $p \in (2,\infty)$ while $\zeta u(\zeta)$ belongs to $L^2$.  Inserting the first right-hand term of \eqref{mu1.exp} into $I_{12}$ gives the product of $\calF\left( \zetabar^{-1} (1-\eta) \dee_\zeta u \right)$ and $\dint e_{-k }\ubar \overline{\mu_2}$. The first factor is the Fourier transform of an $L^2$ function and Lipschitz continuous from $H^{1,1}$ into $L^2$. Thus, it suffices to show that the second factor is a Lipschitz continuous map from $H^{1,1}$ into $L^\infty$. This follows from H\"{o}lder's inequality, \eqref{H11.Lp} with $s=p'$, \eqref{mu.int.bd}, and \eqref{mu.lip.Lp}.

To analyze $I_2$, we use Lemma \ref{lemma:mu.expand} with $N=2$. The term corresponding to $R_{2,N}$ belongs to $L^p$ with appropriate Lipschitz continuity by \eqref{R2.est} and \eqref{R2.lip} together with the fact that $\norm{|u|^2}_{p'}$ is bounded for any $p \in (2,\infty)$ using \eqref{H11.Lp} with $s=2p'$. The remaining terms take the form
\begin{equation}
\label{Ix}
\angles{|u|^2,T^{2j+1} 1} = \frac{1}{2} \angles{|u|^2, P\left( e_k u \left( \overline{T^{2j} 1} \right) \right)} =
		\angles{ e_{-k} w, \overline{T^{2j} 1}}.
\end{equation}
where $w = \ubar \overline{P} \left(|u|^2\right)$ satisfies $\norm{w}_2 \leq C \norm{u}_{H^{1,1}}^3$ owing to \eqref{H11.Lp} and \eqref{HLS}. By Remark \ref{rem:brown}, the form \eqref{Ix} defines a multilinear map from $H^{1,1}$ to $L^2$.
\end{proof}

To finish the proof that $\calR$ is Lipschitz continuous from $H^{1,1}$ to itself, we consider the derivatives $\dee_k r$ and $\dbar_k r$.

\begin{lemma}
\label{lemma:R3}
Fix $M_0>0$.
For any $u,u' \in B_0$, $\nabla (\calR u)$ and $\nabla (\calR u')$ belong to $L^2$ and the estimate
$$ \norm{\nabla (\calR u)  - \nabla (\calR u')}_2 \leq C(M_0) \norm{u-u'}_{H^{1,1}}$$
holds.
\end{lemma}

\begin{proof} By Lemma \ref{lemma:S}, to show Lipschitz continuity of  $u \mapsto \nabla (\calR u)$, it suffices to study the map $u \mapsto \dee_k r$. As usual, we check Lipschitz continuity on $\calS(\bbC)$ and extend by density.

For $u \in \calS(\bbC)$ we compute $\dee_k r = -\calF\left((\dotarg) u(\dotarg)\right)+ I_1+I_2$ where
\begin{align*}
I_1 &= -\frac{1}{\pi} \int e_k(\zeta) \, \zeta u(\zeta) \, \left[ \overline{\mu_1(\zeta,k)}-1 \right] \, dA(\zeta)\\
I_2 &= \frac{1}{2\pi}r(k) \int  e_k(\zeta) u(\zeta) \overline{\mu_2(\zeta,k)} \, dA(\zeta)
\end{align*}
where we used the first equation in \eqref{pre.mu.dbar.k}. To see that $u \mapsto I_1$ is Lipschitz continuous, we may mimic the analysis of $I_1$ in the proof of Lemma \ref{lemma:R2}.
The map $u \mapsto I_2$ defines a Lipschitz continuous map since $u \mapsto r$ is Lipschitz continuous as a map from $H^{1,1}$ to $L^2$ by Lemma \ref{lemma:R1}, $u \in L^{p'}$ by \eqref{H11.Lp}, and $u \mapsto \mu_2$ is Lipschitz continuous from $H^{1,1}$ into $L^p$ by Lemma \ref{lemma:mu.lip}.
\end{proof}

\begin{proof}[Proof of Proposition \ref{prop:R}] An immediate consequence of Lemmas \ref{lemma:R1}-\ref{lemma:R3}.
\end{proof}

The following result is an immediate consequence of Lemma \ref{lemma:IR} and Proposition \ref{prop:R}.
\begin{proposition}
\label{prop:I}
The map $\calI$, initially defined on $\calS(\bbC)$ by \eqref{nu.dbar} and \eqref{I}, extends to $H^{1,1}(\bbC)$. Moreover, for any $M_0>0$, and $r,r' \in B_0$, we have $\calI r, \calI r' \in H^{1,1}(\bbC)$ and
$$ \norm{\calI r - \calI r'}_{H^{1,1}} \leq C(M_0) \norm{r-r'}_{H^{1,1}}. $$
\end{proposition}

\begin{remark}
\label{rem:I}
In analogy to Remark \ref{rem:R}, the extension of $\calI$ to $H^{1,1}(\bbC)$ can be computed as
$$
\left( \calI r \right)(z) = \calF^{-1}(r)(z) + \frac{1}{\pi} \int e_{-k}(z)  r(k) \left(\nu_1(z,k) - 1 \right) \, dA(k)
$$
where the second right-hand integral is absolutely convergent.
\end{remark}

Next, we  show that $\calR$ and $\calI$ are mutual inverses. 

\begin{lemma}
\label{lemma:ur.recon}
Suppose that $u \in H^{1,1}(\bbC)$ and that $r = \calR u$. Let $(\nu_1,\nu_2)$ solve the system \eqref{nu.dbar} with $r = \calR u$. Then
$$
u(z) = \frac{1}{\pi} \int e_{-k}(z) r(k) \nu_1(z,k) \, dA(k).
$$
That is, $\left( \calI \compose \calR \right) u = u$ for all $u \in H^{1,1}(\bbC)$. Similarly, $\left(\calR \compose \calI \right)r =r$ for all $r \in H^{1,1}(\bbC)$. 
\end{lemma}

\begin{proof}
If $\calI \compose \calR$ is the identity map $I$ on $H^{1,1}(\bbC)$, the relation $\calR \compose \calI = I$ is an immediate consequence of \eqref{IR}. Hence, it suffices to show that $\calI \compose \calR$ is the identity map.

The analysis of \S \ref{sec:osc} applies with no essential changes to \eqref{nu.dbar} and shows that this equation has a unique solution for each fixed $z \in \bbC$ and given $r \in H^{1,1}(\bbC)$.   By this uniqueness, the functions $(\nu_1,\nu_2)$ obtained by setting $(\nu_1,\nu_2)=(\mu_1,e_k \overline{\mu_2})$ coincide with the functions $(\nu_1,\nu_2)$ obtained by solving the $\dbar$-problem \eqref{nu.dbar} with $r = \calR u$. 
We will first show that
$$
\lim_{|k| \rarr \infty} \left( \dee_z + k \right) \nu_2 = \frac{1}{2}\ubar
$$
where $\nu_2=e_k \overline{\mu_2}$, $u$ is the given $u \in H^{1,1}(\bbC)$, and the limit is taken in the $L^s(\bbC)$ topology for some $s \in (2,\infty)$. We will then show that, if $\nu_2$ is the solution to \eqref{nu.dbar}, the relation
$$
\lim_{|k| \rarr \infty} \left( \dee_z + k \right) \nu_2 = \frac{1}{2\pi} \int e_{k}(z) \, \overline{r(k) \nu_1(z,k)} \, dA(k) =  \frac{1}{2} \overline{\calI r}
$$
also holds. This proves that $u = \calI r$ in $L^s(\bbC)$. Since $u$ and $\calI r$ belong to $H^{1,1}(\bbC)$, it follows that the equality holds in $H^{1,1}(\bbC)$. 

First, we may compute for each $k \in \bbC$ that
$$
\left( \dee_z + k\right) \nu_2 = e_k \overline{\left(\dbar_z \mu_2\right)} = \frac{1}{2}\ubar \, \mu_1 = 
	\frac{1}{2} \ubar \, \nu_1
$$
as vectors in $L^p(\bbC)$ where we used \eqref{mu.dbar.2}. On the other hand, $\nu_1 -1 = \frac{1}{2} P_k \left( e_k \rbar \, \overline{\nu_2} \right)$ by \eqref{nu.dbar.1} and Lemma \ref{solve}, so that
$$
\left( \dee_z + k \right) \nu_2 - \frac{1}{2}\ubar = 
	\frac{1}{4}\ubar \overline{P_k \left( e_k \rbar \, \overline{\nu_2} \right) }.
$$
For each $k$ and any $s \in (2,\infty)$ we may therefore estimate
$$
\norm{\left(\dee_z + k \right) \nu_2 - \frac{1}{2} \ubar \,}_s 
	\leq	\frac{1}{4} \norm{u}_s \norm{\nu_2}_{C^0(\bbC \times \bbC)} \, \left| P_k (|r|)(k) \right|
$$
The second right-hand factor is bounded owing to \eqref{mu.int.holder.z} since $\nu_2 = e_k \overline{\mu_2}$. The third right-hand factor is a bounded function that vanishes as $|k| \rarr \infty$ by \eqref{vanish}.

Second, from the formula $\nu_2 = \frac{1}{2} P_k \left[ e_k \rbar \overline{\nu_1} \right]$, the fact that $\nu_1=\mu_1$, and \eqref{mu.dbar.1}, we may compute
\begin{align*}
 \left(\dee_z + k \right) \nu_2 - \overline{\calI r}
 	&= \left(\dee_z + k \right) \nu_2 - \frac{1}{2\pi}  \int e_{k} \overline{r(k) \nu_1(z,k)} \, dA(k)\\
	&= \frac{1}{2} P_k \left[ e_{-k} \rbar \overline{ \dbar_z \nu_1 } \right]\\
	&= \frac{1}{4} \ubar P_k \left[ \rbar \mu_2 \right].
\end{align*}
We may then estimate, for each $k \in \bbC$,
$$ \norm{\ubar P_k \left[ \rbar \mu_2 \right]}_s \leq  \norm{u}_s \norm{\mu_2}_{C^0(\bbC \times \bbC)} \left| P_k(|r|) \right|$$
and conclude as before that $\norm{\ubar P_k \left[ \rbar \mu_2 \right]}_s$ vanishes as $k \rarr \infty$. 
\end{proof}

Next, we prove Plancherel-type identities for $\calR$ and $\calI$. 

\begin{lemma}
\label{lemma:Plancherel}
For $u$ and $r$ belonging to $H^{1,1}(\bbC)$, the identities
$$ \norm{\calR u }_2 = \norm{r}_2, \quad \norm{\calI r}_2 = \norm{u}_2 $$
hold.
\end{lemma}

\begin{proof}
We prove the first since the second then follows from \eqref{IR}.  By Lipschitz continuity it suffices to prove the result for $u \in C_0^\infty(\bbC)$. Letting $r = \calR u$ we may compute
\begin{align*}
\int |r(k)|^2 \, dA(k) 
	&= \lim_{R \rarr \infty}
		 \frac{1}{\pi} \int_{|k| \leq R}  \overline{r(k)} 
				\left( \int e_k(\zeta) u(\zeta)  \overline{\mu_1(\zeta,k)} \, dA(\zeta) \right)
			 		\, dA(k) \\
	&=\lim_{R \rarr \infty} 
		\frac{1}{\pi} \int u(\zeta) \left( \int_{|k| \leq R} e_k(\zeta) \overline{r(k)} \overline{\nu_1(\zeta,k)} \, dA(k) \right)  \, dA(\zeta)\\
	&=\lim_{R \rarr \infty}  \left(  I_1(R) + I_2(R) \right)
\end{align*}
where
\begin{align*}
I_1(R) &= \frac{1}{\pi} \int u(\zeta) \left( \int_{|k| \leq R} 
						e_k(\zeta) \overline{r(k)}  \, dA(k) \right)  \, dA(\zeta)\\
I_2(R) &= \frac{1}{\pi} \int u(\zeta) \left( \int_{|k| \leq R} 
						e_k(\zeta) \overline{r(k)} \left[ \overline{\nu_1(\zeta,k)}-1 \right] \, dA(k) \right)  \, dA(\zeta)\
\end{align*}
Since $\dfrac{1}{\pi} \dint_{|k| \leq R} e_k(\zeta)  \overline{r(k)} \, dA(k)$ converges in $L^2$ to $\calF(\rbar)$ we have 
$$
\lim_{R \rarr \infty} I_1(R)= \int u(\zeta) (\calF \rbar)(\zeta) \, dA(\zeta).
$$ 
The analogue of Lemma \ref{lemma:mu.sol} for \eqref{nu.dbar} guarantees that $\nu_1(\zeta,\dotarg)-1 \in L^p(\bbC)$ uniformly in $\zeta \in \bbC$, so that
$$
\lim_{R \rarr \infty} I_2(R) 
	= \frac{1}{\pi} \int u(\zeta) \left( \int e_k(\zeta) \overline{r(k)} \left[\overline{\nu_1(\zeta,k)} -1 \right] \, dA(k) \right)  \, dA(\zeta). 
$$
The Plancherel identity now follows from Remark \ref{rem:I} and the identity $\calF \rbar = \overline{\calF^{-1} r}$. 
\end{proof}

\begin{proof}[Proof of Theorem \ref{thm:lip}] 
An immediate consequence of Propositions \ref{prop:R} and \ref{prop:I} together with Lemmas \ref{lemma:ur.recon} and \ref{lemma:Plancherel}.
\end{proof}

\section{Large-Time Asymptotics}

\label{sec:large.t}

In this section, we prove Theorem \ref{thm:asy} using the formulation \eqref{ISM2} of the inverse scattering method. For $u_0 \in H^{1,1}(\bbC) \cap L^1(\bbC)$, we have $r_0 \in H^{1,1}(\bbC) \cap C_0(\bbC)$ by Remark \ref{rem:r}. In this section we will assume that $r_0 \in H^{1,1}(\bbC) \cap C_0(\bbC)$ and set
$$ \gamma = \norm{r_0}_{H^{1,1}} + \norm{r_0}_{C^0(\bbC)}. $$
Observe that the solution formula \eqref{FSM1} for initial data $v_0 = \calF^{-1} r_0$ may be written
\begin{equation}
\label{FSM2}
v(z,t) = \frac{1}{\pi} \int e^{itS(z,k,t)} r_0(k) \, dA(k),
\end{equation}\
where the real-valued phase function $S$ is given by \eqref{S}.

By \eqref{FSM2}, to prove Theorem \ref{thm:asy}, we need to show that
\begin{equation}
\label{uv.1}
u(z,t)-v(z,t) = \frac{1}{\pi} \int e^{itS(z,k,t)} r_0(k) \left[ \nu_1(z,k,t)-1 \right] \, dA(k) =o\left(t^{-1}\right)
\end{equation}
in $L^\infty_z$-norm, where $\nu_1$ is determined by \eqref{nu.dbar.t}.

To solve the $\dbar$-problem \eqref{nu.dbar.t} and obtain the estimates on $\nu_1-1$ needed to prove \eqref{uv.1}, we introduce the integral operator
\begin{equation}
\label{M}
M\psi = \frac{1}{2}P_k \left( e^{-itS} \overline{ r_0 \psi } \right)
\end{equation}
which depends parametrically on $z$ and $t$ through the phase function $S$. It follows from the theory of \S \ref{sec:osc} that $M$ is a compact operator on $L^p_k(\bbC)$ for each fixed $z,t$ and any $p \in (2,\infty)$, that the resolvent $(I-M^2)^{-1}$ is a bounded operator on $L^p_k(\bbC)$, and
\begin{equation}
\label{nu-1}
 \nu_1 -1 = (I-M^2)^{-1} M^2 1
 \end{equation}
as vectors in $L^p_k(\bbC)$. 
In Lemma \ref{lemma:remainder} we reduce the proof of estimate \eqref{uv.1} to the estimate
\begin{equation}
\label{uv.2}
\frac{1}{\pi} \int e^{itS(z,k,t)} r_0(k) (M^2 1)(z,k,t) \, dA(k) = o\left( t^{-1} \right)
\end{equation}
in $L^\infty_z(\bbC)$ norm. We prove the estimate \eqref{uv.2} in Lemmas \ref{lemma:I1}--\ref{lemma:I4}. For the proofs of Lemmas \ref{lemma:I1}--\ref{lemma:I3}, it suffices to assume that $r \in H^{1,1}(\bbC) \cap C^0(\bbC)$. For Lemma \ref{lemma:I4}, we need to assume that $r \in C_0(\bbC)$. 

We begin with stationary phase estimates on the operator $M$. Recalling \eqref{S} and \eqref{S.crit} we may write
$$S(z,k,t)=4 \Real\left(\left(k-k_c\right)^2\right) + S_0, \quad S_0 = \frac{1}{4}\Real\left(z^2/t^2\right).$$
Hence
$$ S_\kbar (z,k,t) = 4 \left(\kbar - \kbar_c \right). $$
Since $S$ has a single stationary point at $k=k_c$, we introduce a cutoff function 
$$
\chi(k) = \eta \left( t^{1/4} \left(k-k_c\right) \right)
$$
where $\eta \in C_0^\infty(\bbC)$ with $\eta(w)=1$ for $|w| \leq 1$ and $\eta(w)=0$ for $|w| \geq 2$. 
Note that, for any $\sigma \in [1,\infty]$,
\begin{equation}
\label{chi}
\norm{\chi}_\sigma \leq C_\sigma t^{-1/(2\sigma)}.
\end{equation}
We will estimate $M$ by splitting $M= M\chi + M(1-\chi)$, use the small support of $\chi$ to estimate the first term, and the oscillations of the factor $\exp(itS)$ to estimate the second term. 

\begin{lemma}
\label{lemma:ip.t}
Suppose that $p \in (2,\infty)$, that $\psi \in W^{1,p}(\bbC)$, and that $r_0 \in H^{1,1}(\bbC)$. Then, as vectors in $L^p(\bbC)$,
$$
M\left[ (1-\chi) \psi \right] = -\frac{e^{-itS}}{2itS_\kbar} (1-\chi) \overline{r_0\psi} + \frac{1}{2it}P_k \left[ e^{-itS} \dbar_k\left(S_\kbar^{-1} (1-\chi) \overline{r_0} \psibar \right) \right]
$$
\end{lemma}

\begin{proof}
For $\psi,r_0\in C_{0}^{\infty}\left(  {\bbC}\right)  $ this is
a direct consequence of the integration by parts formula \eqref{bug}
with $\varphi$ replaced by $-tS$. Now let $\psi\in W^{1,p}\left(  {\bbC}\right)  $
and $r_0\in H^{1,1}(\bbC)$. If $\left\{  \psi_{n}\right\}  $ and
$\left\{  r_{n}\right\}  $ are sequences from $C_{0}^{\infty}\left(
{\bbC}\right)  $ with $\psi_{n}\rarr\psi$ in $W^{1,p}$ and
$r_{n}\rarr r_0$ in $H^{1,1}(\bbC)$, we have $\psi_{n}\rarr\psi$ in sup
norm so $r_{n}\psi_{n}\rarr r_0\psi$ in $L^{p}\cap L^{2p/\left(
p+2\right)  }$ and $\dbar_{\zeta}\left(  r_{n}\psi_{n}\right)
\rarr\dbar\left(  r_0\psi\right)  $ in $L^{2p/(p+2)}$. Using
\eqref{OP}, we conclude that the identity holds in $L^{p}$-sense for
$\psi\in W^{1,p}$ and $r_0\in H^{1,1}(\bbC)$.
\end{proof}

We'll use the following estimates on singular factors $S_\kbar^{-1}$ and $S_\kbar^{-2}$ that occur in the integrations by parts. We omit the elementary proofs.

\begin{lemma}
For any $\sigma \in (2,\infty]$,
\begin{equation}
\label{S-1}
\norm{S_\kbar^{-1} (1-\chi)}_\sigma \leq C_\sigma t^{1/4-1/(2\sigma)}.
\end{equation}
For any $\sigma \in (1,\infty]$,
\begin{align}
\label{S-2}
\norm{S_\kbar^{-2} (1-\chi)}_\sigma 
	&\leq C_\sigma t^{1/2-1/(2\sigma)}\\
\label{S-3}
\norm{S_\kbar^{-1} (\dbar_k \chi)}_\sigma 
	&\leq C_\sigma t^{1/2-1/(2\sigma)}
\end{align}
\end{lemma}

Using the estimates above we can now estimate $M$ away from points of stationary phase. 

\begin{lemma}
Suppose that $r_0 \in H^{1,1}(\bbC) \cap C^0(\bbC)$. For any $p \in (2,\infty)$, the estimate
\begin{align}
\label{M-2}
\norm{M(1-\chi)\psi}_p			& \leq	C_{p} \gamma t^{-3/4} \left( \norm{\psi}_p + \norm{\dee \psi}_p\right)
\end{align}
holds.
\end{lemma}

\begin{proof}
We will use Lemma \ref{lemma:ip.t}. We compute
\begin{align}
M\left(  1-\chi\right)  \psi 
	&=	\frac{e^{-itS}}{itS_{\kbar}}
				\left(
						1-\chi
				\right) \overline{r_0 \psi}
\label{M-2.pre}\\
	&+	\frac{1}{\pi it}
				\int\frac{e^{-itS}}{k-\zeta}
						\dbar_{\zeta}
							\left(  
									S_{\kbar}^{-1}\left(  1-\chi\right)  \overline{r_0 \psi}
							\right)
				~dA(\zeta)
				\nonumber
\end{align}
so that
\begin{equation}
\left\Vert M\left(  1-\chi\right)  \psi\right\Vert _{p}
		\leq C_{p}
				t^{-1}
				\left( 
						 \sum_{j=0}^{4}
						 		\left\Vert J_{j}\right\Vert_{p}
			    \right) .
\label{M-2.pre-2}%
\end{equation}
Here, $J_{0}$ is $t$ times the first right-hand term in
(\ref{M-2.pre}). The terms $J_{1}$, $J_{2}$, $J_{3}$, $J_{4}$ are $t$ times the
integrals that arise in the second right-hand term of (\ref{M-2.pre}) by
applying the product rule to
\begin{align}
\label{eq:J1-J4}
\dbar_{\zeta}
	\left(  S_{\kbar}^{-1}\left(  1-\chi\right)  \overline{r_0} \overline{\psi}\right)   
	&  =	-4S_{\kbar}^{-2}\left(  1-\chi\right)
				\overline{r_0} \overline{\psi}-S_{\kbar}^{-1}\left(  \dbar\chi\right)
				\overline{r_0} \overline{\psi}\\
	\nonumber
	&  +	S_{\kbar}^{-1}\left(  1-\chi\right)  
				\dbar \overline{r_0} \overline{\psi}+S_{\kbar}^{-1}\left(  1-\chi\right)  
				\overline{r_0} \overline{\partial\psi}.
\end{align}
By \eqref{HLS}, to estimate $\| J_i \|_p$ for $i=1,2,3,4$, we must estimate the $L^{2p/(p+2)}$ norms of each of the four right-hand terms in \eqref{eq:J1-J4}.

$J_0$: Using H\"{o}lder's inequality and \eqref{S-1}, we estimate
$$
\left\Vert J_{0}\right\Vert _{p} 
		 \leq \| S_{\kbar}^{-1} (1-\chi) \overline{r_0} \psibar \|_p 
		  \leq C_{p}t^{\frac{1}{4}}
			\gamma \| \psi \|_p
$$
which shows that $J_0$ is estimated by a constant times $t^{1/4}$. 

$J_1$, $J_2$: We estimate $J_1$ since the estimate for $J_2$ is similar. Using \eqref{S-2}, we have
$$
\left\Vert J_{1}\right\Vert _{p}  
		  \leq C_p \left\| 4S_{\kbar}^{-2}(1-\chi) \overline{r_0} \psibar \right\|_{2p/(p+2)} 
		  \leq C_{p,\sigma}
			t^{\frac{1}{4}}
				\gamma \left\Vert \psi\right\Vert _{p}
$$
where in the last step we used H\"{o}lder's inequality,  \eqref{S-2} with $\sigma=2$, and the bound $\|r_0\|_\infty \leq \gamma$.  In the estimate for $J_2$, we replace \eqref{S-2} by \eqref{S-3}. 

$J_3$, $J_4$:
To estimate $J_3$, we use \eqref{S-1} and H\"{o}lder's inequality to conclude that
\begin{align*}
\left\Vert J_{3}\right\Vert _{p}  
		&	\leq C_p \left\| S_{\kbar}^{-1} (1-\chi) \overline{\dee r_0} \, \psibar \right\|_{2p/(p+2)}\\
		&  \leq C_{p} \| S_\kbar^{-1} (1-\chi)\|_{\sigma_1} \, 
				\| \dee r_0 \|_2 \, \| \psi \|_{\sigma_2}\\
		& \leq C_p \gamma t^{1/4-1/(2\sigma_1)} \| \psi \|_{\sigma_2}.
\end{align*}
Here $\sigma_1^{-1}+\sigma_2^{-1}=p^{-1}$. If $\sigma_2=p$ we may take $\sigma_1=\infty$. Hence, we can estimate $J_3$ in all cases by a constant times $t^{1/4}$. The estimate for $J_4$ is similar, with $\| \dee r_0 \|_2$ replaced by $\| r_0 \|_2$ in the estimates.  Recalling \eqref{M-2.pre-2} and combining these estimates leads to
\eqref{M-2}.
\end{proof}

We will make use of the following estimates on $M$.

\begin{lemma}
Suppose that $r_0 \in H^{1,1}(\bbC) \cap C^0(\bbC)$. For any $p$ with $p >2$, the following estimates hold.
\begin{align}
\label{M-0}
\norm{M}_{\calBbar(L^p)} 		&\leq		C_p \gamma,\\
\label{M-3}
\norm{M\chi}_{\calBbar(L^p)}	&\leq		C_p \gamma t^{-1/4},\\
\label{M-9}
\norm{M^2}_{\calB(L^p)}		&\leq		C_p t^{-1/4} \gamma^2, \\
\label{M-4}
\norm{M\chi 1}_p 				& \leq	C_p \gamma t^{-1/4-1/(2p)}, \\
\label{M-5}
\norm{M(1-\chi)1}_p 			& \leq	C_p \gamma t^{-3/4},\\
\label{M-8}
\norm{M^4 1}_p					& \leq	C_p \gamma^4 t^{-1-1/(2p)}.
\end{align}
\end{lemma}

\begin{proof}
The estimate \eqref{M-0} is an immediate consequence of \eqref{OP}. 

To prove \eqref{M-3}, we use \eqref{OP} to estimate
$$ \norm{M\chi \psi}_p \leq C_p \norm{r_0\chi}_2 \norm{\psi}_p \leq C_p \gamma \norm{\chi}_2 \norm{\psi}_p$$
and use \eqref{chi} with $\sigma=2$.

To prove \eqref{M-9}, we use \eqref{M-2}  and \eqref{M-3} to estimate
$$
\norm{M \varphi}_p 	\leq	C_p \gamma \left(t^{-1/4}\norm{\varphi}_p 
									+ t^{-3/4}\left( \norm{\varphi}_p+ \norm{\dbar \varphi}_p \right) \right)
$$
where in the last term we used $\norm{\dee \varphi}_p \leq C_p \norm{\dbar \varphi}_p$ owing to Lemma \ref{lemma:S}. Setting $\varphi=M\psi$ and using the estimate above, \eqref{M-0}, and the trivial estimate $\norm{\dbar M\psi}_p \leq \gamma \norm{\psi}_p$ we obtain \eqref{M-9}.

To prove \eqref{M-4}, we use \eqref{HLS} to estimate $\norm{M\chi 1}_p \leq C_p \gamma \norm{\chi}_{2p/(p+2)}$ and apply \eqref{chi}.

To prove \eqref{M-5}, we trace through the proof of \eqref{M-2} with $\psi=1$.

To prove \eqref{M-8}, we first note that
\begin{equation}
\label{M8.1}
\norm{M1}_p \leq C_p \gamma t^{-1/4-1/(2p)}
\end{equation}
by \eqref{M-4} and \eqref{M-5}. Next, from \eqref{M-2} and \eqref{M-3}, the estimate
\begin{equation}
\label{M8.2}
\norm{M\psi}_p \leq C_p \gamma t^{-1/4} \left( \norm{\psi}_p + t^{-1/2} \norm{\dbar \psi}_p \right) 
\end{equation}
holds, 
where in the last term we used Lemma \ref{lemma:S}. Starting with \eqref{M8.1} and iterating with \eqref{M8.2} we see that 
$$\norm{M^j 1}_p \leq C_p \gamma^j t^{-j/4-1/(2p)}.$$
The case $j=4$ gives \eqref{M-8}.
\end{proof}

From \eqref{M-9} it follows that for each $p \in (2,\infty)$,  there is a $T=T(\gamma,p)>0$ so that 
\begin{equation}
\label{M.res}
\sup_{t>T(\gamma,p)} \norm{(I-M^2)^{-1}}_{\calB(L^p)} \leq 2.
\end{equation}

\begin{lemma}
\label{lemma:remainder}
Suppose that $r_0 \in H^{1,1}(\bbC) \cap C^0(\bbC)$. Then, the estimate
$$
\sup_{z \in \bbC} 
	\left|	u(z,t)-v(z,t) - \frac{1}{\pi}\int e^{itS} r_0 M^2 1 \right|
\leq
C(p,\gamma) t^{-1-1/(2p)}
$$
holds for any $p \in (2,\infty)$ and all $t>T(\gamma,p)$.
\end{lemma}

\begin{proof}
From the first equality in \eqref{uv.1}, \eqref{nu-1}, and the identity $(I-M^2)^{-1} -I - M^2 = (I-M^2)^{-1}M^4$ we conclude that 
$$
u(z,t)-v(z,t) - \frac{1}{\pi}\int e^{itS} r_0 M^2 1  = \frac{1}{\pi} \int e^{itS} r_0 (I-M^2)^{-1} M^4 1
$$
The result now follows from H\"{o}lder's inequality, the fact that  $r_0 \in L^{p'}(\bbC)$ for any $p \in (2,\infty)$, \eqref{M-8}, and \eqref{M.res}.
\end{proof}

It remains to estimate
\begin{equation}
\label{eq:M2.decomp}
\int e^{itS}rM^{2}1=I_{1}+I_{2}+I_{3}+I_{4} 
\end{equation}
where%
\begin{align*}
I_{1}  
	&  =\int e^{itS}r_0	M\left[  \left(  1-\chi\right)  M\left(  \chi\right) \right]  ,\\
I_{2}  
	&  =\int e^{itS}rM\left[  \chi M\left(  1-\chi\right)  \right]  ,\\
I_{3}  
	&  =\int e^{itS}rM\left[  \left(  1-\chi\right)  M\left(1-\chi\right)  \right] \\
I_{4}  
	&  =\int e^{itS}rM\left[  \chi M\left(  \chi\right)  \right]  .
\end{align*}

First, we analyze the mixed terms $I_{1}$ and $I_{2}$. In each integral we
will split $e^{itS}r_0=e^{itS}r_0\chi+e^{itS}r_0\left(  1-\chi\right)  $ and bound each of the terms separately. To control the first type of term, we will use the estimate
\begin{equation}
\left\vert \int e^{itS}r_0\chi\psi\right\vert 
	\leq C \gamma t^{1/(2p)-1/2} \| \psi \|_p
\label{eq:I12.in}
\end{equation}
true for any $p \in (2,\infty)$.
To control the second type of term, we will use the integration by parts formula%
$$
\int e^{itS}r_0\left(  1-\chi\right)  \psi
	=-\frac{1}{it}\int e^{itS}
			\dbar_{k}
			\left(  
					S_{\kbar}^{-1}r_0
					\left(  1-\chi\right)\psi
			\right). 
$$
Expanding the $\dbar_k$-derivative into four terms, using H\"{o}lder's inequality, and applying the inequalities \eqref{S-1}--\eqref{S-3}, we conclude that for any $p \in (2,\infty)$,
\begin{align}
\label{eq:I12.out.est}
\left\vert 
	\int e^{itS}r_0\left(  1-\chi\right)  \psi
\right\vert  
	&  \leq	
		C_p \gamma t^{1/(2p)-1} \left\Vert \psi\right\Vert_{p}
	  +C_p \gamma t^{1/(2p)-1}\| \psi \|_p
\\
	&\quad  
		+C_p\gamma t^{1/(2p)-1} \|\psi\|_p
 		+C_p\gamma t^{1/(2p)-1} \| \dbar \psi \|_p
\nonumber
\\[0.2cm]
	&\leq C_ p \gamma t^{1/(2p)-1} \left(\| \psi \|_p + \| \dbar \psi \|_p \right).
\nonumber
\end{align}
%where in the last line we used the boundedness of the transform $\calS^* = \dee \dbar^{-1}$ on $L^p$ (see remarks following Lemma \ref{lemma:S}).

First, we consider $I_{1}$.

\begin{lemma}
\label{lemma:I1}
Suppose that $r_0\in H^{1,1}(\bbC) \cap C^0(\bbC)$, $p \in (2,\infty)$, and $t>1$. Then
\begin{equation}
\label{eq:I1.est}
\left| \int e^{itS}rM\left[  (1-\chi)M\chi\right] \right|  
	\leq C_p \gamma^3 t^{1/(2p)-5/4}. 
\end{equation}

\end{lemma}

\begin{proof}
We split $I_{1}=J_{1}+J_{2}$ where%
\begin{align*}
J_{1}  &  =\int e^{itS}\chi rM\left[  (1-\chi)M\chi\right]  ,\\
J_{2}  &  =\int e^{itS}\left(  1-\chi\right)  rM\left[  (1-\chi)M\chi\right].
\end{align*}

To bound $J_1$, we use \eqref{eq:I12.in} with $\psi=M(1-\chi)M\chi$ and \eqref{M-2} with $s=p$ to estimate
\begin{align*}
\left| J_1 \right|
	&	\leq	C_p\gamma t^{1/(2p)-1/2} \| M(1-\chi) M \chi \|_p	\\
	&	\leq	C_p \gamma^2 t^{1/(2p)-5/4} 
					\left(\|M\chi\|_p+\|\dbar(M\chi)\|_p \right)\\
	&	\leq	C_p \gamma^3 t^{-5/4}
\end{align*}
where in the last step we used $\|\dbar (M\chi)\|_p = \| r_0 \chi \|_p$ and \eqref{chi}. 

To bound $J_2$, we \eqref{eq:I12.out.est} with $\psi=M(1-\chi)M\chi$, the operator bound \eqref{M-0}, and \eqref{M-4} to estimate
\begin{align*}
\left| J_2 \right| 
	&	\leq 	C_p \gamma t^{1/(2p)-1} 
					\left(
				 		\|M(1-\chi)M\chi\|_p+\|\dbar M(1-\chi)M\chi \|_p 
					\right) \\
	&	\leq	C_p \gamma^3 t^{1/(2p)-5/4}
\end{align*}
Combining these two estimates completes the proof.
\end{proof}

Next, we consider $I_{2}$.

\begin{lemma}
\label{lemma:I2}
Suppose that $r_0\in H^{1,1}(\bbC)\cap L^\infty(\bbC)$, $t>1$, and  $p \in (2,\infty)$.  Then
\begin{equation}
\label{eq:I2.est}
\left| \int e^{itS}rM\left[  \chi M\left(  1-\chi\right)  \right] \right| 
		\leq C_p \gamma^3 t^{1/(2p)-5/4}.
\end{equation}
\end{lemma}

\begin{proof}
As before we write $I_{2}=J_{1}+J_{2}$ where%
\begin{align*}
J_{1}  &  =\int e^{itS}r_0\chi M\left[  \chi M\left(  1-\chi\right)  \right]
,\\
J_{2}  &  =\int e^{itS}r_0\left(  1-\chi\right)  M\left[  \chi M\left(
1-\chi\right)  \right]  .
\end{align*}

To estimate $J_1$, we use \eqref{eq:I12.in} with $\psi = M \chi M(1-\chi)$, \eqref{M-0}, and \eqref{M-5} to estimate
\begin{align*}
\left| J_1 \right|
	&	\leq	C_p \gamma^2 t^{1/(2p)-1/2} \| \chi M(1-\chi) \|_p \\
	& 	\leq	C_p \gamma^3 t^{1/(2p)-5/4} 
\end{align*}

To estimate $J_2$, we use \eqref{eq:I12.out.est} with $\psi=M\chi M (1-\chi)$, \eqref{M-0}, and \eqref{M-5}  to estimate
\begin{align*}
\left| J_2 \right| 
	&	\leq	C_p \gamma t^{1/(2p)-1} \| M\chi M(1-\chi)\|_p+
				\| \dbar M\chi M(1-\chi) \|_p \\
	&	\leq	C_p \gamma^2 t^{1/(2p)-1} \|M(1-\chi)\|_p \\
	&	\leq	C_p \gamma^3 t^{1/(2p)-7/4}.
\end{align*}

Combining these estimates completes the proof.
\end{proof}

Next, we bound $I_{3}$.

\begin{lemma}
\label{lemma:I3}
Suppose $r_0\in H^{1,1}(\bbC) \cap C^0(\bbC)$,  $t>1$, and $p \in (2,\infty)$. Then
\begin{equation}
\label{eq:I3.est}
\left\vert 
	\int e^{itS}rM
			%\left[  
					\left(  1-\chi\right)  
					M	\left[  
							\left(1-\chi\right)  
						\right]  
			%\right] 
\right\vert 
\leq
C_p \gamma^3 t^{1/(2p)-5/4}
\end{equation}

\end{lemma}

\begin{proof}
First, we insert $1= \chi + (1-\chi)$ and write the integral to be estimated as $J_1+J_2$ where
\begin{align*}
J_1
	&	=	\int e^{itS} \chi r_0 M\left[ (1-\chi) M \left[ (1-\chi) \right] \right] \, dA \\
J_2 
	&	=	\int e^{itS} (1-\chi)  r_0 M\left[ (1-\chi) M \left[ (1-\chi) \right] \right] \, dA 
\end{align*}

We estimate, using \eqref{eq:I12.in} with $\psi=M(1-\chi)\left[M(1-\chi)\right]$, 
\begin{align*}
\left| J_1 \right|
	&	\leq	C_p \gamma t^{1/(2p)-1/2} 
			\| M(1-\chi)M(1-\chi) \|_p\\
	&	\leq 	C_p \gamma^3 t^{1/(2p)-5/4}	
\end{align*}
where in the last step we used \eqref{M-5} and \eqref{M-0}.

To estimate $J_2$, we use \eqref{eq:I12.out.est} with $\psi =
M(1-\chi)M(1-\chi)$, \eqref{M-0}, and \eqref{M-5} to conclude that
\begin{align*}
\left| J_2 \right|
	&	\leq	C_p \gamma t^{1/(2p)-1}
		\left( \| M(1-\chi)M(1-\chi) \|_p+\gamma \|M(1-\chi)\|_p \right)\\
	& \leq	C_p \gamma^3 t^{1/(2p)-7/4}.
\end{align*}

Combining these two estimates completes the proof.
\end{proof}

Finally, we show that $I_{4}$ is $o\left(  t^{-1}\right)$. Recall that, for $u_0 \in L^1(\bbC)$, $r_0 \in C_0(\bbC)$ by Remark \ref{rem:r}.

\begin{lemma}
\label{lemma:I4}
Suppose $r_0\in H^{1,1}(\bbC) \cap C_0(\bbC)$. Then
\begin{equation}
\label{eq:I4.est}
\lim_{t\rarr  +\infty}
	t\int e^{itS}r_0~M
		\left(  \chi
			\left(  M
				\left(\chi1\right)  
			\right)  
		\right)  =0. 
\end{equation}
\end{lemma}

\begin{proof}
We first write
$$ t \int e^{itS} r_0 M\left(\chi \left(M \left( \chi 1 \right) \right) \right) = J_1 + J_2 $$
where
\begin{align*}
J_1 &=	t \int e^{itS} \chi r_0 \psi_M  \\
J_2 &=  t \int e^{itS} \left(1-\chi\right) r_0 \psi_M
\end{align*}
and 
$$ \psi_M = M\left(\chi \left(M \left( \chi 1 \right) \right) \right). $$

We first show that $J_2 \rarr 0$ as $t \rarr \infty$ using the estimate \eqref{eq:I12.out.est} with $\psi=\psi_M$. We obtain, for any $p \in (2,\infty)$,
\begin{align*}
\left| J_2 \right| 
	&	\leq C_p \gamma t^{1/(2p)}
			\left(\| \psi_M \|_p+ \| \dbar \psi_M\|_p \right) \\
	& 	\leq	C_p \gamma t^{1/(2p)} 
			\left(\| \psi_M \|_p + \gamma \| \chi M(\chi) \|_p\right) \\
	&	\leq C_p \gamma^3 t^{-1/4}
\end{align*}
In the second step, we used \eqref{M-0}, and in the last step, we used \eqref{M-4}. This shows that $J_2 \rarr 0$ as $t \rarr \infty$.

We now turn to $J_1$.  Let us write $\tilde{k}$ for $k-k_c$. We may compute
\[
J_1  =	\frac{t e^{itS_{0}}}{4\pi^2}
		\int_{\bbC^{3}}
				\frac{e^{
							4it\operatorname{Re}
							\left(  \tilde{k}^{2}-(\tilde{k}^{\prime})^{2}+(\tilde{k}^{^{\prime\prime}})^{2}\right) 
							}
						G(k,k^{\prime},k^{\prime\prime})}
						{	\left(  k-k^{\prime}\right)  
							\left(  \overline{k^{\prime}}-\overline{k^{\prime\prime}}\right)
						}
						~dA(k,k^{\prime},k^{\prime\prime})~
\]
where
\[
G(k,k^{\prime},k^{\prime\prime})=r_0(k)\overline{r_0(k^{\prime})}r_0(k^{\prime
\prime})\chi(k)\chi(k^{\prime})\chi(k^{\prime\prime})
\]
Define $\zeta=t^{1/4}\left(  k-k_{c}\right)  $ and similarly for
$\zeta^{\prime}$ and $\zeta^{\prime\prime}$.  We see that the expression $J_1$ is given
by%
\[
I(z,t)=\frac{e^{itS_{0}}}{4\pi^2}
	\int
			\frac{e^{4it^{1/2}
							\operatorname{Re} \left(  
									\zeta^{2}-\zeta^{\prime2}+\zeta^{\prime\prime2}			
							\right)  }
							H(\zeta,\zeta^{\prime},\zeta^{\prime\prime})}{
					\left(  \zeta-\zeta^{\prime}\right)  
					\left(  
							\overline{\zeta^{\prime}}-
							\overline{\zeta^{\prime\prime}}
					\right)  }
					dA(\zeta,\zeta^\prime,\zeta^{\prime\prime})
\]
where%
%\begin{align*}
$$
H(\zeta,\zeta^{\prime},\zeta^{\prime\prime})  
		%&  
		=\eta(\zeta)
			\eta\left(\zeta^{\prime}\right) 
			 \eta(\zeta^{\prime\prime})
		%\\
		%&  
		\times 
				r_0(k_{c}+\zeta/\sqrt[4]{t})~
				\overline{r_0(k_{c}+\zeta^{\prime}/\sqrt[4]{t})}~
				r_0(k_{c}+\zeta^{\prime\prime}/\sqrt[4]{t}).
$$
%\end{align*}
Clearly, $\left| I(z,t) \right| \leq C \|r_0 \|_{C_0(\bbC)}^3$, where $C$ is bounded uniformly in $z$ and $t$, and $I(z,t)$ is a continuous multilinear function of $r_0 \in C_0(\bbC)$. 
 Thus, to show that $\lim_{t\rarr\infty
}I(z,t)=0$, it suffices to check for $r_0\in C_{0}^{\infty}\left(
\bbC\right)  $ since such $r_0$ are dense in $C_0(\bbC)$.
For such $r_0$, we have%
\begin{align*}
I(z,t)  &  =e^{itS_{0}}
	\left
		\vert r_0\left(  k_{c}\right)  \right\vert^{2}r_0(k_{c})
		\int
				\frac{e^{4it^{1/2}
						\operatorname{Re}
						\left(  \zeta^{2}-\zeta^{\prime2}
							+\zeta^{\prime\prime2}
						\right)  }
						\eta(\zeta) \eta\left(  \zeta^{\prime}\right)
						\eta(\zeta^{\prime\prime})}{
						\left(  
							\zeta-\zeta^{\prime}
						\right)  
						\left(
							\overline{\zeta^{\prime}}-
							\overline{\zeta^{\prime\prime}}\right)  }%
				dA(\zeta,\zeta^\prime,\zeta^{\prime\prime})\\
&  +\mathcal{O}\left(  t^{-1/4}\right)  .
\end{align*}
Consider now the integral
$$ J(t)=\int \frac{e^{4it^{1/2}
				\operatorname{Re}
				 \left(
				 		\zeta^2 - \zeta^{\prime 2} + \zeta^{\prime \prime 2}
				 \right)}
				 \eta(\zeta) \eta(\zeta') \eta(\zeta^{\prime \prime})
				 }{
				 (\zeta-\zeta')(\overline{\zeta'}-\overline{\zeta^{\prime \prime}})
				 }
				 \,
				 dA(\zeta,\zeta',\zeta'').
$$
The integrand is an $L^1(\bbC^3)$ function owing to the compact support of $\eta$. The integral $J(t)$ is thus a special case of the integral
$$ J(t;f) = \int_{\bbC^3} 
						e^{4it^{1/2} \operatorname{Re}
										\left(
				 								\zeta^2 - \zeta^{\prime 2} + \zeta^{\prime \prime 2}
				 						\right)} 
				 		f(\zeta,\zeta^\prime, \zeta^{\prime\prime})
				 	\, dA(\zeta,\zeta^\prime,\zeta^{\prime\prime}). $$
It suffices to show that $J(t,f) \rarr 0$ as $t \rarr \infty$. Owing to the trivial bound $|J(t,f)| \leq \| f \|_1$, it suffices to do 
so for a dense set of $f \in L^1(\bbC^3)$.  We first observe that finite linear combinations of compactly 
supported product functions
of the form $g_1(\zeta) g_2(\zeta') g_3(\zeta'')$ are dense in $L^1(\bbC^3)$, so it suffices to show that
$$ \lim_{t \rarr \infty} \int e^{4it^{1/2} \operatorname{Re} \zeta^2} g(\zeta) \, dA(\zeta) = 0. $$
Now write $\zeta = \zeta_1 + i \zeta_2$ and note that, by a further density argument, we may take $g(\zeta) = h_1(\zeta_1) h_2(\zeta_2)$. As $\operatorname{Re} (\zeta^2)= \zeta_1^2 - \zeta_2^2$ it now suffices to show that
$$ \lim_{t \rarr \infty} \int e^{\pm 4it^{1/2} s^2} h(s) \, ds = 0 $$
for bounded and compactly supported $f$. This is now an easy consequence of the Riemann-Lebesgue lemma and a simple change of variables.

\end{proof}

\begin{proof}
[Proof of Theorem \ref{thm:asy}]An immediate consequence of
Lemma \ref{lemma:remainder}, (\ref{eq:M2.decomp}), and the estimates
(\ref{eq:I1.est}), (\ref{eq:I2.est}), (\ref{eq:I3.est}), and (\ref{eq:I4.est}).
\end{proof}

\appendix{}

\section{Multilinear Estimates by Michael Christ}

\label{sec:christ}

In this appendix we establish a rather general multilinear inequality in terms
of weak type Lebesgue spaces, then specialize it to deduce the inequality of
Brown \cite{Brown:2001} stated in Proposition~\ref{brown}.

Let ${\mathbb{F}}$ be one of the two fields ${\mathbb{F}}={\bbR}$ or
${\mathbb{F}}={\bbC}$, equipped with Lebesgue measure in either case.
Consider ${\bbC}$--valued multilinear functionals
\begin{equation}
\Lambda(f_{1},f_{2},\cdots,f_{m})=\int_{{\mathbb{F}}^{N}}\prod_{j=1}^{m}%
f_{j}(\ell_{j}(y))\,dy \label{defn}%
\end{equation}
where each $\ell_{j}:{\mathbb{F}}^{N}\rarr{\mathbb{F}}^{N_{j}}$ is a
surjective $\mathbb{F}$-linear transformation, $f_{j}:{\mathbb{F}}^{N_{j}%
}\rarr{\bbC}$, and $dy$ denotes Lebesgue measure on ${\mathbb{F}%
}^{N}$. A complete characterization of those exponents $(p_{1},\cdots
,p_{m})\in\lbrack1,\infty]^{m}$ for which there are inequalities of the form
\begin{equation}
\big|\Lambda(f_{1},f_{2},\cdots,f_{m})\big|\leq C\prod_{j=1}^{m}\Vert
f_{j}\Vert_{L^{p_{j}}} \label{maininequality}%
\end{equation}
has been obtained in \cite{BCCT2}. Such an inequality implicitly includes the
assertion that the integral \eqref{maininequality} converges absolutely
whenever each $f_{j}$ belongs to $L^{p_{j}}$. To review this result, we first
recall key definitions from \cite{BCCT1},\cite{BCCT2}.

Denote by $\operatorname{dim}_{\mathbb{F}}(V)$ the dimension of a vector space
$V$ over ${\mathbb{F}}$. Throughout the discussion, ${\mathbb{F}}$ should be
considered as fixed; vector spaces, subspaces, and linear mappings are defined
with respect to ${\mathbb{F}}$.

\begin{definition}
Relative to a set of exponents $\{p_{j}\}$, a subspace $V\subset{\mathbb{F}%
}^{N}$ is said to be critical if
\begin{equation}
\operatorname{dim}_{\mathbb{F}}(V)= \sum_{j} p^{-1}_{j} \operatorname{dim}%
_{\mathbb{F}}(\ell_{j}(V)),
\end{equation}
to be supercritical if the right-hand side is strictly less than
$\operatorname{dim}_{\mathbb{F}}(V)$, and to be subcritical if the right-hand
side is strictly greater than $\operatorname{dim}_{\mathbb{F}}(V)$.
\end{definition}

Throughout the discussion, the reciprocal of any infinite exponent is
interpreted as $0$. The subspace $\{0\}$ is always critical.

\begin{theorem}
\label{thm:bcct} \cite{BCCT2} Let ${\mathbb{F}}= {\bbR} $ or
${\mathbb{F}}={\bbC}$. Let $N\ge1$ and $N_{j}\ge1$ for all
%% change [] to {} twice
$j\in \{1,2,\cdots,m\}$. For each index $j\in\{1,2,\cdots,m\}$ let $\ell
_{j}:{\mathbb{F}}^{N}\to{\mathbb{F}}^{N_{j}}$ be an ${\mathbb{F}}$--linear
surjective mapping. Let $p_{j}\in[1,\infty]$. Then \eqref{maininequality}
holds if and only if ${\mathbb{F}}^{N}$ is critical relative to $\{p_{j}\}$
and no proper subspace of ${\mathbb{F}}^{N}$ is supercritical relative to
$\{p_{j}\}$.
\end{theorem}

This theorem was stated in \cite{BCCT2} only for ${\mathbb{F}}={\bbR}$,
but the proof given in \cite{BCCT2} applies equally well to ${\mathbb{F}%
}={\bbC}$. See also \cite{BCCT1} for a different proof and more thorough
analysis for the case ${\mathbb{F}}={\bbR}$.

In order to extend this theorem to include Brown's inequality
(\ref{ineq:brown}), we will utilize the Lorentz spaces $L^{p,r}$ as defined
for instance in \cite{SW:1971}. These spaces are defined for $(p,r)\in
\lbrack1,\infty)\times\lbrack1,\infty]$, and are Banach spaces except in the
exceptional case $(p,r)=(1,1)$. Throughout the following discussion, we assume
that $(p,r)$ is not equal to $(1,1)$. The facts needed about the Lorentz
spaces for our discussion are these:

\begin{enumerate}
\item[(i)] $L^{p,p}$ equals the Lebesgue space $L^{p}$.

\item[(ii)] $L^{p,\infty}$ equals weak $L^{p}$. That is, $f\in L^{p,\infty
}({\mathbb{F}}^{n})$ if and only if there exists $C_{f}<\infty$ such that for
every $\alpha\in(0,\infty)$, $\big|\{x\in{\mathbb{F}}^{n}:|f(x)|>\alpha
\}\big|\leq C_{f}^{p}\alpha^{-p}$. Here $|E|$ denotes the Lebesgue measure of
a subset $E$ of ${\mathbb{F}}^{n}$. The infimum of all such $C_{f}$ is denoted
by $\Vert{f}\Vert_{L^{p,\infty}}$. This quantity is not in general a norm, but
is equivalent to one unless $(p,r)=(1,1)$; see \cite{SW:1971}.

\item[(iii)] In particular, the functions $|x|^{-d/p}$ and $|z|^{-2d/p}$
belong to $L^{p,\infty}( {\bbR} ^{d})$ and to $L^{p,\infty}({\bbC%
}^{d})$, respectively.

\item[(iv)] If $r \ge p$ then $\norm{f}_{L^{p,r}} \le C \norm{f}_{L^p}$ for all functions $f$, where $C<\infty$ depends only on $p,r$.

\end{enumerate}

The next result extends Theorem~\ref{thm:bcct} to Lorentz spaces, although
perhaps not in the most definitive manner.

\begin{theorem}
\label{thm:lorentzbcct} Let ${\mathbb{F}}= {\bbR} $ or ${\mathbb{F}%
%% change [] to {} per referee
}={\bbC}$. Let $N\ge1$ and $N_{j}\ge1$ for all $j\in\{1,2,\cdots,m\}$. For
%% change [] to {} per referee
each index $j\in\{1,2,\cdots,m\}$ let $\ell_{j}:{\mathbb{F}}^{N}\to{\mathbb{F}%
}^{N_{j}}$ be an ${\mathbb{F}}$--linear surjective mapping. Let each exponent
$p_{j}$ belong to the open interval $(1,\infty)$.

Suppose that with respect to $\{p_{j}\}$, the total space ${\mathbb{F}}^{N}$
is critical, and every nonzero proper subspace of ${\mathbb{F}}^{N}$ is
subcritical. Then for all exponents $r_{j}\in\lbrack1,\infty]$ satisfying
\begin{equation}
\sum_{j}r_{j}^{-1}=1,
\end{equation}
and for all functions $f_j \in L^{p_j,r_j}(\mathbb{F}^{N_j})$,
$\prod_{j=1}^{m}f_{j}\circ\ell_{j}$ belongs to $L^{1}({\mathbb{F}}^{N})$.
Moreover, there exists $C<\infty$ independent of $\{f_{j}\}$ such that
\begin{equation}
\big|\Lambda(f_{1},f_{2},\cdots,f_{m})\big|\leq C\prod_{j=1}^{m}\Vert
f_{j}\Vert_{L^{p_{j},r_{j}}}. \label{lorentzinequality}%
\end{equation}

\end{theorem}

The proof will utilize the following crude multilinear interpolation theorem,
established in \cite{Christ:1985}.

\begin{proposition}
\label{prop:interpolate} Let $a_{j}\in\lbrack0,\infty)$, and suppose that at
least one of these numbers is nonzero. Let $\Omega=\{(t_{1},\cdots,t_{j}%
)\in(0,1)^{m}:\sum_{j}a_{j}t_{j}=1\}$, equipped with the topology induced by
its embedding in $(0,1)^{m}$. Let $(X,{\mathcal{A}},\mu)$ be any measure
space. Let $\Lambda=\Lambda(f_{1},\cdots,f_{m})$ be a complex-valued
multilinear form defined for all $m$-tuples of measurable simple functions
$f_{j}:X\rarr{\bbC}$.

Let ${\mathcal{O}}$ be a nonempty open subset of $\Omega$. Suppose that for
each $t=(t_{1},\cdots,t_{m})\in{\mathcal{O}}$ there exists $C_{t}<\infty$ such
that
\begin{equation}
|\Lambda(f_{1},\cdots,f_{m})|\leq C_{t}\prod_{j}\Vert{f_{j}}\Vert_{L^{p_{j}%
,1}}\qquad\text{ where $p_{j}=t_{j}^{-1}$}%
\end{equation}
for all $m$-tuples of simple functions $f_{j}$. Then for any relatively
compact subset ${\mathcal{O}}^{\prime}\subset{\mathcal{O}}$ there exists
$C<\infty$ such that for all $t\in{\mathcal{O}}^{\prime}$ and all exponents
$r_{j}$ satisfying $\sum_{j=1}^{m}r_{j}^{-1}=1$, for all $m$-tuples of
measurable simple functions,
\begin{equation}
|\Lambda(f_{1},\cdots,f_{m})|\leq C\prod_{j}\Vert{f_{j}}\Vert_{L^{p_{j},r_{j}%
}}\qquad\text{ where $p_{j}=t_{j}^{-1}$}.
\end{equation}

\end{proposition}

\begin{proof}
[Proof of Theorem~\ref{thm:lorentzbcct}]It suffices to apply
Theorem~\ref{thm:bcct} and Proposition~\ref{prop:interpolate} in combination.
Indeed, if an $m$-tuple $p=\{p_{j}: 1\le j\le m\}$ satisfies the hypotheses of
Theorem~\ref{thm:lorentzbcct}, then so does any $m$-tuple $q=\{q_{j}: 1\le
j\le m\}$ satisfying the equation $\sum_{j} \frac{N_{j}}{N} q_{j}^{-1}=1$ such
that each $q_{j}^{-1}$ is sufficiently close to $p_{j}^{-1}$. Indeed, as $V$
varies over all nonzero proper subspaces of ${\mathbb{F}}^{N}$, the numbers
$\sum_{j} p_{j}^{-1}\frac{\operatorname{dim}_{\mathbb{F}}(\ell_{j}%
(V))}{\operatorname{dim}_{\mathbb{F}}(V)}$ take on finitely many values, and
are all strictly greater than one by the subcriticality hypothesis. Therefore
these strict inequalities continue to hold whenever $q$ is sufficiently close
to $p$. The hypotheses of Proposition~\ref{prop:interpolate} are thus
satisfied. Applying that Proposition yields inequality \eqref{lorentzinequality}.
\end{proof}

Consider now the multilinear inequality of Brown \cite{Brown:2001}. Let%
\[
\Lambda_{n}(\rho,q_{0},q_{1},\ldots,q_{2n})=\int_{\bbC^{2n+1}}%
\frac{\left\vert \rho(\zeta)\right\vert \left\vert q_{0}(z_{0})\right\vert
\ldots\left\vert q(z_{2n})\right\vert }{\prod_{j=1}^{2n}\left\vert
z_{j-1}-z_{j}\right\vert }d\mu(z)
\]
where $d\mu(z)$ is product measure on $\bbC^{2n+1}$ and $\zeta
=\sum_{j=0}^{2n}(-1)^{j}z_{j}$. The inequality states that
\begin{equation}
\left\vert \Lambda_{n}(\rho,q_{0},q_{1},\ldots,q_{2n})\right\vert \leq
C\left\Vert \rho\right\Vert _{2}\prod_{j=0}^{2n}\left\Vert q_{j}\right\Vert
_{2}. \label{RBinequality}%
\end{equation}
Note that since $\Lambda_{n}$ is multilinear, it follows directly from this
statement that the map%
\[
\left(  \rho,q_{0},q_{1},\ldots,q_{2n}\right)  \mapsto\Lambda_{n}(\rho
,q_{0},q_{1},\ldots,q_{2n})
\]
is Lipschitz continuous from any bounded subset of $\left(  L^{2}%
({\bbC}\right)  )^{2n+1}$ to $\bbC$.

To deduce \eqref{RBinequality} from Theorem~\ref{thm:lorentzbcct}, set
${\mathbb{F}}= {\bbC}$, $N=2n+1$, and $m=4n+2$. Let the index $j$ range
%% change notation per referee
over $[0,4n+1]$, set $N_{j}=1$ for all $j\in\{0,\cdots,4n+1\}$, write
$z=(z_{0},\cdots,z_{2n})$, and consider the linear functionals $l_{j}:
{\bbC}^{2n+1}\to{\bbC}^{1}$ defined by
\begin{equation}
l_{j}(z)=
\begin{cases}
z_{j} & \text{for $0\le j\le2n$,}\\
z_{j-2n}-z_{j-2n-1} \qquad & \text{for $2n< j\le4n$},\\
\sum_{i=0}^{2n} (-1)^{i} z_{i} & \text{for $j=4n+1$}.
\end{cases}
\end{equation}

The following linear algebraic fact will be proved below.

\begin{lemma}
\label{lemma:verify} The $(4n+2)$-tuple of exponents $p=(p_{j})=(2,2,\cdots
,2)$ satisfies the hypotheses of Theorem~\ref{thm:lorentzbcct}.
\end{lemma}

%% changed per MC
To apply the lemma to inequality \eqref{RBinequality}, set $r_j=(2n+2)^{-1}$ for $0 \le j \le 2n$ and for $j=4n+1$, and $r_j=
\infty$ for $2n<j\le 4n$.  For each $j\in(2n,4n]$,
define $f_{j}:{\bbC}^{1}\rarr{\bbR}^{+}$ by $f_{j}%
(w)=|w|^{-1}$. Each of these functions belongs to $L^{2,\infty}({\bbC%
}^{1})$. The factors $|z_{j}-z_{j-1}|^{-1}$ appearing in \eqref{RBinequality}
are then $|z_{j}-z_{j-1}|^{-1}=f_{j}(\ell_{j}(z))$, for $j\in(2n,4n]$. Setting
$f_{j}=q_{j}$ for all $j\in\lbrack0,2n]$ and $f_{4n+1}=\rho$, $\Lambda
_{n}(\rho,q_{0},\cdots,q_{2n})$ equals $\Lambda(f_{0},\cdots,f_{4n+1}%
)=\int_{{\bbC}^{N}}\prod_{j=0}^{4n+1}f_{j}(\ell_{j}(z))\,dz$. By Theorem~\ref{thm:lorentzbcct} in conjunction with Lemma \ref{lemma:verify}, 
\begin{equation*}
\left| \Lambda_n(\rho,q_0,\dots,q_{2n}) \right| \le C \norm{\rho}_{2,r} \prod_{j=0}^{2n} \norm{q_j}_{L^{2,r}}
\end{equation*}
where $r=2n+2$. Since $2n+2 \ge 2$, the $L^{2,r}$ norm is majorized by a constant multiple of the $L^2$ norm.
%Inequality
%\eqref{RBinequality} therefore follows from Theorem~\ref{thm:lorentzbcct}
%together with Lemma~\ref{lemma:verify}. 
\qed

This reasoning yields various refinements of \eqref{RBinequality}. For
instance, any one of the functions $\rho,q_{j}$ may be taken to be in
$L^{2,\infty}( {\bbC})$ rather than in $L^{2}$.

\begin{proof}
[Proof of Lemma~\ref{lemma:verify}]Firstly, $N=2n+1$, while
\[
\sum_{j=0}^{4n+1} p_{j}^{-1} \operatorname{dim}_{\bbC}(\ell
_{j}({\bbC}^{N}))=\sum_{j=0}^{4n+1} p_{j}^{-1}N_{j} = \sum_{j=0}^{4n+1}
2^{-1}\cdot1 = 2^{-1}\cdot(4n+2) = N.
\]
Thus ${\mathbb{F}}^{N}$ is critical relative to $(2,2,\cdots,2)$.

It remains to show that any nonzero proper complex subspace $V$ of
${\bbC}^{N}$ is subcritical. For any index $j$, since $\ell_{j}$ is a
linear mapping from ${\bbC}^{N}$ to ${\bbC}^{1}$, either
$\operatorname{dim}_{\bbC}(\ell_{j}(V))=1$, or $\ell_{j}$ vanishes
identically on $V$. Let $S$ be the set of all $j\in[0,\cdots,2n]$ such that
$z_{j}\equiv0$ for all $z=(z_{0},\cdots,z_{2n})\in V$, and let $T$ be the set
of all $j\in[1,2n]$ such that $z_{j}-z_{j-1}\equiv0$ for all $z\in V$, but
neither $j$ nor $j-1$ belongs to $S$.

%% j\in [0,2n] changed to j\in [1,2n] per referee
The mapping $l_{j+2n}:V\to{\bbC}$ is surjective if $j\in[1,2n]$ and
$j\notin T\cup S$. For if not, then it vanishes identically; $z_{j}-z_{j-1}=0$
for all $z\in V$. Since $j\notin T$, the definition of $T$ forces at least one
of the indices $j,j-1$ to belong to $S$, that is, at least one of the
functions $z\mapsto z_{j}$ and $z\mapsto z_{j-1}$ vanishes identically on $S$.
The equation $z_{j}-z_{j-1}\equiv0$ then forces both of these functions to
vanish identically. Therefore both indices $j,j-1$ belong to $S$,
contradicting the hypothesis that $j\notin T\cup S$.

%% j \in (2n,4n] changed to j\in(0,2n] per referee
A further consequence is that the number of $j\in(0,2n]$ such that $j\notin
T$, but $z_{j}-z_{j-1}\equiv0$ for all $z\in V$, is at most $|S|-1$. Equality
occurs if and only if $S=[k,k-1+|S|]$ for some $k\in[0,2n]$.

%% changed the set of mappings per referee
The set of mappings $\{\ell_{j}: j\in S\} \cup \{\ell_{j+2n}:j \in T \}$ is linearly independent, and
$V$ is contained in the intersection of their nullspaces, so
$\operatorname{dim}_{\bbC}(V)\le2n+1-|S|-|T|$. On the other hand,
\begin{align*}
\sum_{j=0}^{4n+1}2^{-1}  &  \operatorname{dim}_{\bbC}(\ell_{j}(V))\\
&  = \sum_{j=0}^{2n}2^{-1}\operatorname{dim}_{\bbC}(\ell_{j}(V)) +
\sum_{j=2n+1}^{4n}2^{-1}\operatorname{dim}_{\bbC}(\ell_{j}(V)) +
2^{-1}\operatorname{dim}_{\bbC}(\ell_{4n+1}(V))\\
&  \ge2^{-1} (2n+1-|S|) + 2^{-1}(2n-|T|-(|S|-1)) + 2^{-1}\operatorname{dim}%
_{\bbC}(\ell_{4n+1}(V))\\
&  = \big(2n+1-|S|-|T|\big)+2^{-1}|T| + 2^{-1}\operatorname{dim}_{\bbC%
}(\ell_{4n+1}(V))\\
&  \ge\operatorname{dim}_{\bbC}(V)+2^{-1}|T| + 2^{-1}\operatorname{dim}%
_{\bbC}(\ell_{4n+1}(V)).
\end{align*}
This is strictly greater than $\operatorname{dim}_{\bbC}(V)$ unless
$T=\emptyset$, $V$ is contained in the nullspace of $\ell_{4n+1}$,
$\operatorname{dim}_{\bbC}(V)= 2n+1-|S|$, and $S=[k,k-1+|S|]$ for some
$k\in[0,2n]$ with $k-1+|S|\le2n$.

Suppose that $T=\emptyset$, and that $V$ is contained in the nullspace of
$\ell_{4n+1}$. $S$ cannot be all of $[0,2n]$, for this would force $V=\{0\}$,
contrary to hypothesis. Therefore the equation $\ell_{4n+1}|_{V}\equiv0$ is
not forced by the equations $\ell_{j}|_{V}\equiv0$ for all $j\in S$, so
$\operatorname{dim}_{\bbC}(V)$ must be strictly less than $2n+1-|S|$.
Therefore $\sum_{j=0}^{4n+1}2^{-1}\operatorname{dim}_{\bbC}(\ell
_{j}(V))$ is strictly greater than $\operatorname{dim}_{\bbC}(V)$ in all
cases; every nonzero proper subspace of ${\bbC}^{N}$ is subcritical.
\end{proof}

\section{Time Evolution of Scattering Maps}

\label{sec:time}

The purpose of this appendix is to give a self-contained proof that the
function $u$ defined by (\ref{ISM1}) solves the DS\ II\ equation for
$u_{0}\in\calS({\bbC})$. Previous proofs may be found, for
example, in the papers of Beals-Coifman \cite{BC:1985,BC:1989,BC:1990} and
Sung \cite{Sung:1994}, Part III. We suppose that $r\in C^{1}(\bbR%
_{t};\calS({\bbC}))$ obeys a linear equation%
\[
\dot{r}=i\varphi r
\]
where $\varphi$ is a real-valued polynomial in $k$ and $\kbar$. We will
obtain an effective formula for $\dot{u}$ if $u=\calI(r)$ by
differentiating%
\[
u=\left\langle e_{k}\rbar,\nu_{1}\right\rangle
\]
and exploiting solutions $\left(  \nu_{1}^{\#},\nu_{2}^{\#}\right)  $ to a
`dual' problem
\begin{align}
\dbar_{k}\nu_{1}^{\#}  &  =\frac{1}{2}e_{k}\overline{r^{\#}%
}\overline{\nu_{2}^{\#}}\label{nu.dbar.dual}\\
\dbar_{k}\nu_{2}^{\#}  &  =\frac{1}{2}e_{k}\overline{r^{\#}%
}\overline{\nu_{1}^{\#}}\nonumber
\end{align}
where $r^{\#}=\rbar$. The following lemma on symmetries of the map
$\calR$ shows that $r^{\#}=\calR(u^{\#})$ where $u^{\#}%
(z)=\overline{u(-z)}$.

\begin{lemma}
Let $u,u^{\flat}\in H^{1,1}(\bbC)$ and let $r=\calR\left(  u\right)  $,
$r^{\flat}=\calR\left(  u^{\flat}\right)  $. 
\newline(i) If $u^{\flat
}(z)=-u(z)$, then $r^{\flat}(k)=-r(k)$,
\newline(ii) if $u^{\flat}(z)=-u(-z)$,
then $r^{\flat}(k)=-r(-k)$, and
\newline(iii) if $u^{\flat}(z)=\overline{u}%
(z)$, then $r^{\flat}(k)=-\overline{r(k)}$.
\end{lemma}

\begin{proof}
In what follows we let $\left(  \mu_{1}^{\flat},\mu_{2}^{\flat}\right)  $
denote the solutions to (\ref{mu.dbar}) with $u$ replaced by $u^{\flat}$.

(i) follows from (\ref{R}) and the fact that $\mu_{1}^{\flat}=\mu_{1}$.

(ii) follows from (\ref{R}) and the fact that $\mu_{1}^{\flat}%
(z,k)=\mu_{1}(-z,-k)$

(iii) From the definition (\ref{R}) we compute (recall
(\ref{pairing}))%
\begin{align*}
r^{\flat}(k)  &  =\left\langle e_{-k}\overline{u},\overline{\mu_{1}^{\flat}%
}\right\rangle \\
&  =\left\langle e_{-k}u,\left(  I-\overline{P}_{k}e_{-k}uP_{k}e_{k}%
\overline{u}\right)  ^{-1}1\right\rangle \\
&  =\left\langle \left(  I-e_{-k}u\overline{P}_{k}\overline{u}e_{k}%
P_{k}\right)  ^{-1}e_{-k}u,1\right\rangle \\
&  =\left\langle \left(  I-\overline{P}_{k}\overline{u}e_{k}P_{k}%
e_{-k}u\right)  ^{-1}1,e_{k}\overline{u}\right\rangle \\
&  =\overline{r(-k)}%
\end{align*}
as claimed.
\end{proof}

From the formula%
\[
\left[  \partial_{t},T_{k}^{2}\right]  =-\frac{i}{4}P_{k}e_{k}\overline
{r}\left[  \varphi,\overline{P}_{k}\right]  e_{-k}r
\]
we have%
\begin{align*}
\dot{\nu}_{1}  &  =\left[  \partial_{t},\left(  I-T_{k}^{2}\right)
^{-1}\right]  1\\
&  =-\frac{i}{4}\left(  I-T_{k}^{2}\right)  ^{-1}P_{k}e_{k}\rbar\left[
\varphi,\overline{P}_{k}\right]  e_{-k}r\nu_{1}%
\end{align*}
so that%
\begin{align*}
\dot{u}  &  =i\left\langle e_{k}\varphi\rbar,\nu_{1}\right\rangle
-\frac{i}{4}\left\langle e_{k}\rbar,\left(  I-T_{k}^{2}\right)
^{-1}P_{k}e_{k}\rbar\left[  \varphi,\overline{P}_{k}\right]  e_{-k}%
r\nu_{1}\right\rangle .\\
&  =i\left\langle e_{-k}rf_{1},\varphi g_{1}\right\rangle +i\,\left\langle
f_{2},\varphi e_{-k}r\nu_{1}\right\rangle
\end{align*}
where%
\begin{align*}
f_{1}  &  =\overline{P}_{k}\left(  I-\left(  T_{k}^{2}\right)  ^{\ast}\right)
^{-1}e_{k}\rbar,\\
g_{1}  &  =\overline{P}_{k}e_{-k}r\nu_{1},\\
f_{2}  &  =1+e_{-k}r\overline{P}_{k}\left(  I-\left(  T_{k}^{2}\right)
^{\ast}\right)  ^{-1}e_{k}\rbar.
\end{align*}
Noting that $\left(  T_{k}^{2}\right)  ^{\ast}=\frac{1}{4}e_{k}\overline
{r}\overline{P}_{k}e_{-k}rP$, it is not difficult to see that%
\begin{align*}
f_{1}(z,k)  &  =\overline{\nu_{2}^{\#}(-z,k)},\\
g_{1}(z,k)  &  =\overline{\nu_{2}(z,k)},\\
f_{2}(z,k)  &  =\nu_{1}^{\#}(-z,k),
\end{align*}
so that%
\[
\dot{u}(z,t)=i\left\langle e_{-k}r\overline{\nu_{2}^{\#}(-z,\dotarg)}%
,\varphi\overline{\nu_{2}(z,\dotarg)}\right\rangle +i\left\langle \nu_{1}%
^{\#}(-z,\dotarg),\varphi e_{-k}r\nu_{1}(z,\dotarg)\right\rangle
\]
where we have suppressed the $t$-dependence of $\nu$ and $\nu^{\#}$. Setting
\[
\eta(z,k)=\frac{1}{2}e_{k}(z)\rbar(k)\nu_{2}^{\#}(-z,k)\overline
{\nu_{2}(z,k)}+\frac{1}{2}\overline{\nu_{1}^{\#}(-z,k)}e_{-k}(z)r(k)\nu
_{1}(z,k)
\]
we have%
\[
\dot{u}(z)=2i\int\varphi(k)\eta(z,k)~dA(k).
\]
Using (\ref{nu.dbar})\ and (\ref{nu.dbar.dual}), we can write%
\[
\eta(z,k)=\dbar_{k}\left[  \nu_{2}^{\#}(-z,k)\nu_{1}%
(z,k)\right]
\]
and%
\[
\overline{\eta(z,k)}=\dbar_{k}\left[  \nu_{1}^{\#}(-z,k)\nu
_{2}(z,k)\right]
\]
so that, if $\varphi(k)=4\Real\left(  k^{2}\right)  $, we conclude
that%
\[
\dot{u}(z)=4i\left(  I_{1}+\overline{I_{2}}\right)
\]
(the complex conjugate on $I_{2}$ is intentional)\ where%
\begin{align*}
I_{1}  &  =\int k^{2}\dbar_{k}\left[  \nu_{2}^{\#}(-z,k)\nu
_{1}(z,k)\right]  ~dA(k),\\
I_{2}  &  =\int k^{2}\dbar_{k}\left[  \nu_{1}^{\#}(-z,k)\nu
_{2}(z,k)\right]  ~dA(k).
\end{align*}
The integrands in $I_{1}$ and $I_{2}$ are exact differentials and, for
$r\in\calS({\bbC})$, vanish rapidly at infinity. We can evaluate
$I_{1}$ and $I_{2}$ using the fact that, if $h$ is a smooth function with
$\dbar h$ of rapid decay and%
\begin{equation}
h\sim\sum_{j\geq0}\frac{h_{j}}{k^{j+1}} \label{eq:h.exp}%
\end{equation}
then%
\[
\int k^{n}\dbar_{k}h~dA(k)=2\pi ih_{n}.
\]
We compute the large-$k$ asymptotic expansions of $\nu_{1}$ and $\nu_{2}$ in
Appendix \ref{sec:asy}. Write $\left[  h\right]  _{j}$ for $h_{j}$ in the
expansion (\ref{eq:h.exp}). In terms of the expansion (\ref{eq:nu.asy}) we
have
\begin{align*}
\left[  \nu_{2}^{\#}(-z,k)\nu_{1}(z,k)\right]  _{2}  &  =\nu_{2,0}^{\#}%
\nu_{1,2}+\nu_{2,1}^{\#}\nu_{1,1}+\nu_{2,2}^{\#}\nu_{1,0},\\
\left[  \nu_{1}^{\#}(-z,k)\nu_{2}(z,k)\right]  _{2}  &  =\nu_{2,0}^{\#}%
\nu_{1,2}+\nu_{2,1}^{\#}\nu_{1,1}+\nu_{2,2}^{\#}\nu_{1,0}%
\end{align*}
where $\nu^{\#}$ corresponds to the potential $u^{\#}$, and, since $\nu^{\#}$
is evaluated at $-z$, we replace $u$ by $-\overline{u}$, $P$ by $-P$, and
$\partial$ by $-\partial$ in (\ref{eq:nu.start}) and (\ref{eq:nu.11}%
)-(\ref{eq:nu.22}) to find the expansion coefficients for $\nu^{\#}$.
Straightforward computation using (\ref{eq:nu.start}) and (\ref{eq:nu.11}%
)-(\ref{eq:nu.22}) gives%
\[
\left[  \nu_{2}^{\#}(-z,k)\nu_{1}(z,k)\right]  _{2}=\frac{1}{4}u\left(
\calS\left(  \left\vert u\right\vert ^{2}\right)  \right)  -\frac{1}%
{2}\partial^{2}u
\]
where we used the identity $\left(  \dbar^{-1}f\right)
^{2}=2\dbar^{-1}\left(  f\dbar^{-1}f\right)  $
with $f=\left\vert u\right\vert ^{2}$ to eliminate terms of fifth order in
$u$. Similarly,%
\[
\left[  \nu_{1}^{\#}(-z,k)\nu_{2}(z,k)\right]  _{2}=-\frac{1}{4}u\left(
\calS\left(  \left\vert u\right\vert ^{2}\right)  \right)  +\frac{1}%
{2}\partial^{2}\overline{u}.
\]
Finally, we obtain%
\[
i\dot{u}(z)=-2(\partial^{2}u+\dbar^{2}u)-u\left(  g+\overline
{g}\right)
\]
where%
\[
g=-\calS\left(  \left\vert u\right\vert ^{2}\right)  .
\]
This is exactly the DS\ II\ equation.

\section{Asymptotic Expansions}

\label{sec:asy}

In this section we compute large-parameter asymptotic expansions of the
solutions $\nu=\left(  \nu_{1},\nu_{2}\right)  $ of (\ref{nu.dbar}).
Exploiting the fact that $\nu=\left(  \mu_{1},e_{k}\overline{\mu_{2}}\right)
$, we conclude from (\ref{mu.dbar}) that
\begin{align}
\dbar_{z}\nu_{1}  &  =\frac{1}{2}u\nu_{2}\label{nu.dbar.z}\\
\left(  \partial_{z}+k\right)  \nu_{2}  &  =\frac{1}{2}\overline{u}\nu
_{1}\nonumber
\end{align}
For $r\in\calS({\bbC})$, the functions $\left(  \nu_{1},\nu
_{2}\right)  $ admit a large-$k$ asymptotic expansion of the form%
\begin{equation}
\nu\sim\left(  1,0\right)  +\sum_{\ell\geq0}k^{-(\ell+1)}\nu^{(\ell)}
\label{eq:nu.asy}%
\end{equation}
where $\nu^{(\ell)}=\left(  \nu_{1,\ell},\nu_{2,\ell}\right)  ^{T}$. From the
system (\ref{nu.dbar.z}) we easily deduce that
\begin{equation}
\nu_{1,0}=\frac{1}{4}\dbar^{-1}\left(  \left\vert u\right\vert
^{2}\right)  ,~~\nu_{2,0}=\frac{1}{2}\overline{u} \label{eq:nu.start}%
\end{equation}
while for $\ell\geq1$,%
\begin{align*}
\nu_{2,\ell}  &  =\frac{1}{2}\overline{u}\nu_{1,\ell-1}-\partial\nu_{2,\ell
-1}\\
\nu_{1,\ell}  &  =\frac{1}{2}P\left(  u\nu_{2,\ell}\right)  .
\end{align*}
It easily follows that
\begin{align}
\nu_{1,1}  &  =\frac{1}{16}P\left(  \left\vert u\right\vert ^{2}P\left(
\left\vert u\right\vert ^{2}\right)  \right)  -\frac{1}{4}P\left(
u\partial\overline{u}\right)  ,\label{eq:nu.11}\\
\nu_{2,1}  &  =\frac{1}{8}\overline{u}P\left(  \left\vert u\right\vert
^{2}\right)  -\frac{1}{2}\partial\overline{u},\label{eq:nu.21}\\
\nu_{2,2}  &  =\frac{1}{32}\overline{u}~P\left(  \left\vert u\right\vert
^{2}P\left(  \left\vert u\right\vert ^{2}\right)  \right) \label{eq:nu.22}\\
&  \frac{1}{8}\partial\left(  \overline{u}~P\left(  \left\vert u\right\vert
^{2}\right)  \right)  -\frac{1}{8}\overline{u}P(u\partial\overline{u}%
)+\frac{1}{2}\partial^{2}\overline{u}.\nonumber
\end{align}

\begin{remark}
In a similar way one can show that for $r\in\calS%
({\bbC})$, $\mu$ has a large-$z$ asymptotic expansion whose
coefficients are computed in terms of $r$ and its derivatives. Thus for
example
\begin{align*}
\mu_{1}(z,k)  
		&  =	1+\frac{1}{z}\left(  \frac{1}{4}\dbar_{k}^{-1}
					\left(  \left\vert r\right\vert ^{2}\right)  \right)  
				+\mathcal{O}\left(\left\vert z\right\vert ^{-2}\right)  ,\\
\mu_{2}(z,k)  
		&  =	\frac{1}{z}\left(  \frac{1}{2}r\right)  +\
				\mathcal{O}\left(	\left\vert z\right\vert ^{-2}  \right)  .
\end{align*}

\end{remark}

\end{document}